\documentclass[1p]{elsarticle}

\usepackage{lineno,hyperref}

\journal{arXiv}
 \usepackage{amsmath, amsthm, amssymb, amsfonts}
\usepackage{graphicx}
\usepackage{pdflscape}
\usepackage{rotating}
\usepackage{epstopdf}
\usepackage{algorithmic}
\usepackage{yfonts,mathrsfs,makecell,enumerate}
\ifpdf
  \DeclareGraphicsExtensions{.eps,.pdf,.png,.jpg}
\else
  \DeclareGraphicsExtensions{.eps}
\fi
\usepackage{amsopn}

\newcommand{\dd}{ { \mathrm{d}} }

\newcommand{\R}{\mathbb {R}}
\newcommand{\D}{\mathcal {D}}
\newcommand{\BC}{\mathrm{LR}}

\newtheorem{theorem}{Theorem}
\newtheorem{proposition}[theorem]{Proposition}
\newtheorem{corollary}[theorem]{Corollary}
\newtheorem{lemma}[theorem]{Lemma}

\newdefinition{definition}[theorem]{Definition}
\newdefinition{remark}[theorem]{Remark}

\begin{document}

\begin{frontmatter}

\title{FRACTIONAL PARTIAL DIFFERENTIAL EQUATIONS WITH BOUNDARY CONDITIONS}
\fntext[marsden]{Baeumer and Kov\'acs were partially funded by the Marsden Fund administered by the Royal Society of New Zealand.}

\author{Boris Baeumer\fnref{marsden}}
\address{University of Otago, New Zealand}\ead{bbaeumer@maths.otago.ac.nz}

\author{Mih\'aly Kov\'acs\fnref{marsden}}
\address{Chalmers University of Technology, Sweden}
\ead{mihaly@chalmers.se}

\author{Harish Sankaranarayanan\fnref{hari}}
\address{Michigan State University, USA}\fntext[hari]{Sankaranarayanan was supported by ARO MURI grant W911NF-15-1-0562.}
\ead{harish@stt.msu.edu}

\begin{abstract}
  We identify the stochastic processes associated with one-sided fractional partial differential equations on a bounded domain with various boundary conditions.  This is essential for modelling using spatial fractional  derivatives. We show well-posedness of the associated Cauchy problems in $C_0(\Omega)$ and $L_1(\Omega)$. In order to do so we develop a new method of embedding finite state Markov processes into Feller processes and then show convergence of the respective Feller processes. This also gives a numerical approximation of the solution. The proof of well-posedness closes a gap in many numerical algorithm articles approximating solutions to fractional differential equations that use the Lax-Richtmyer Equivalence Theorem to prove convergence without checking well-posedness. 
\end{abstract}


\end{frontmatter}


\section{Introduction}
\label{introduction}
The Fokker-Planck equation of a L\'evy stable process on $\mathbb R$ is a fractional-in-space partial differential equation. The (spatial) fractional derivative operator employed therein is non-local with infinite reach. In this article we investigate one-sided fractional derivative operators on a bounded interval $\Omega$ with various boundary conditions. We also identify the stochastic processes whose marginal densities are the fundamental solutions to the corresponding fractional-in-space partial differential equations. 
That is, we show convergence of easily identifiable (sub)-Markov processes (that are essentially finite state),  to a (sub)-Markov process governed by a Fokker-Planck equation on a bounded interval where the spatial operator is a truncated fractional derivative operator with appropriate boundary conditions. 
This is achieved using the Trotter-Kato Theorem \cite[p. 209]{Engel2000}, regarding convergence of Feller semi-groups on $C_0 (\Omega)$ and strongly continuous positive contraction semigroups on $L_1(\Omega)$, and hence showing process convergence \cite[p. 331, Theorem 17.25]{Kallenberg1997}. As a by-product this closes a gap in  numerical algorithm articles approximating solutions to fractional in space differential equations that use the Lax-Richtmyer Equivalence Theorem to prove convergence without checking well-posedness of the problem. 

On the PDE side this article extends ideas of \cite{Zhang2007b,Castillo-Negrete2006} for no-flux boundary conditions and \cite{Podlubny2009,Du2012} and others for Dirichlet boundary conditions. We link them to ideas on reflected and otherwise modified stochastic stable processes  of \cite{Patie2017,Patie2012,Baeumer2015}. 

A significant part of this article can be found in modified and extended form in the Ph.D. thesis of Harish Sankaranarayanan, \cite{Sankaranarayanan2014}.

\subsection{Preliminaries and notation for stochastic processes} 
As we are going to embed finite state Markov processes within Feller processes we start with some basic properties of finite state \emph{sub-Markov processes}. 

A finite state sub-Markov process $(X^n_t)_{t \geq 0}\subset\{1,\ldots,n\}$ is uniquely determined by its \emph{transition rate matrix} $G_{n\times n}$. The diagonal entries $g_{i,i}\le 0$ denote the total rate at which particles leave state $i$ and the entries $g_{i,j}\ge 0,\;i\neq j$ denote the rate at which particles move from state $i$ to state $j$. Furthermore, for each $i$, $-g_{i,i}\ge \sum_{j\neq i} g_{i,j}$, with the difference being the rate at which particles are removed from the system if they are in state $i$.  

The transition rate matrix is the \emph{infinitesimal generator} of the stochastic process. It is also the generator of a semigroup of linear operators $(S(t))_{t \geq 0}$ on $\ell^n_\infty=(\R^n,\|.\|_\infty)$ satisfying the system of ODEs (called the backwards equation)
$$S'(t)\mathbf f=G_{n\times n}S(t)\mathbf f; S(0)\mathbf f=\mathbf f$$ or $$ S(t)\mathbf f=e^{tG_{n \times n}}\mathbf f.$$
If we let $f(j)=\mathbf f_j$, then $$ \left(S(t)\mathbf f\right)_i=\mathbb{E}[f(X^n_t)|X^n_0=i];$$i.e, the $i$-component of $S(t)\mathbf f$ is the expectation of $f(X^n_t)$ conditioned on $X^n_0=i$. In particular, for  a target $\mathbf f=\mathbf e_j$, at each coordinate $i$, $\left(S(t)\mathbf f\right)_i$ is the probability that the process is now in state $j$ given that it started $t$ time units earlier in state $i$. This is why $\left(S(t)\right)_{t\ge 0}$ is called the \emph{backwards semigroup} of the process $(X^n_t)_{t \geq 0}$.

This is in contrast to the \emph{forward semigroup} $(T(t))_{t \geq 0}$ acting on $\mathbf g\in \ell^n_1=(\R^n,\|.\|_1)$ with $T(t):\mathbf g\mapsto \mathbf g^T e^{tG_{n \times n}}=e^{tG_{n \times n}^*}\mathbf g$, or $T(t)\mathbf g$ being the unique solution to the Fokker-Planck equation
$$T'(t)\mathbf g=G^*_{n\times n}T(t)\mathbf g;\quad T(0)\mathbf g=\mathbf g.$$ Here $G_{n \times n}^*$ is the adjoint of $G_{n \times n}$ and $\mathbf g^T$ the transpose of $\mathbf g$. Now $T(t)\mathbf g$ is the probability distribution of $(X^n_t)_{t \geq 0}$ given that the initial probability distribution of $X^n_0$ is $\mathbf g$.  In particular, if $\mathbf g=\mathbf e_j$, $\left(T(t)\mathbf g\right)_i$ is the probability that $X^n_t=i$ given that $X^n_0=j$. 


Recall (e.g., \cite{Bottcher2013,Kolokoltsov2011}) that these concepts are extendable to general sub-Markov processes taking values in a locally compact separable metric space $\Omega$ with the backward semigroup $(S(t))_{t \geq 0}$ acting on bounded functions
\begin{equation}\label{backward}
S(t)f(x)=\mathbb{E}[f(X_t)|X_0=x]
\end{equation}
 and the forward semigroup $(S^*(t))_{t \geq 0}$ acting on regular Borel measures on $\Omega$, 
$$S^*(t)\mu=\mbox{distribution of }X_t|X_0\mbox{ has distribution } \mu.$$
Let  $$C_0(\Omega)=\overline{C_c^\infty(\Omega)}^{\|\cdot\|_\infty},$$ the closure of the space of smooth functions with compact support within $\Omega$ under the supremum norm. Note that if $\Omega$ is an open set then $C_0(\Omega)$ is the space of continuous functions that converge to zero at the boundary and if $\Omega$ is compact, then $C_0(\Omega)$ is the space of continuous functions on $\Omega$. 
If $(X_t)_{t \geq 0}$ is a sub-Markov process such that its backwards semigroup $(S(t))_{t \geq 0}$ is a positive, strongly continuous contraction semigroup that leaves $C_0(\Omega)$ invariant, then $(X_t)_{t \geq 0}$ is called a \emph{Feller process}. On the other hand, for any positive, strongly continuous, contraction semigroup $(S(t))_{t \geq 0}$ on $C_0(\Omega)$ (called a Feller semigroup) there exists a Feller process $(X_t)_{t \geq 0}$ with $(S(t))_{t \geq 0}$ as its backwards semigroup \cite[Chapter 17]{Kallenberg1997}.

For Feller processes the forward semigroup $\left(S^*(t)\right)_{t\ge 0}$ is the adjoint semigroup and in many cases leaves $L_1(\Omega)$, considered as a subspace of the Borel measures, invariant (identifying Borel measures with their densities, if they exist).  In this case the restriction of $S^*$ to $L_1(\Omega)$,
\begin{equation}\label{forward} T(t)f= \mbox {density of }X_t|X_0\mbox{ has density } f\end{equation}  is strongly continuous on $L_1(\Omega)$. 

The connection to differential equations comes from the fact that each semigroup has a generator $(A_S, \D(A_S))$ or $(A_T,\D(A_T))$ with $$\D(A_S)=\left\{f:\lim_{h\to 0}\frac{S(h)f-f}h=:A_Sf\mbox{ exists}\right\}$$ and $\D(A_T)$ defined analogously. It can be shown that the domains $\D(A_S)$ and $\D(A_T)$ are dense and the operators are linear, closed, and dissipative (usually (pseudo-)differential operators). Furthermore, the semigroups are the unique solution operators to the initial value problems on the respective Banach spaces
$$u'(t)=Au(t); u(0)=f.$$
We also call these equations the \emph{governing (backward ($A=A_S$) or forward ($A=A_T$)) differential equations}. Boundary conditions are encoded in $\Omega$ or in the definition of $\D(A)$.

\subsection{The spectrally positive stable process on $\mathbb R$ and fractional derivatives}

Let $Y_t$ be the spectrally positive $\alpha$-stable process for some $1<\alpha< 2$; i.e., the L\'evy process with Fourier transform
$$\mathbb E[e^{-ikY_t}]= e^{t(ik)^\alpha}.$$
This zero mean process has the property of being a pure positive jump process that has negative drift; i.e. $Y_{\inf}(t)=\inf_{0\le \tau\le t} Y_\tau$ is continuous and $Y_{\max}(t)=\max_{0\le \tau\le t} Y_\tau$ is piecewise constant \cite{Bertoin1996a}. 
Its governing backward equation is
$$u'(t)=D_-^\alpha u(t); u(0)=u_0\in C_0(\R)$$ and the governing forward equation (Fokker-Planck) is 
$$u'(t)=D_+^\alpha u(t); u(0)=u_0\in L_1(\R),$$
where the \emph{fractional derivatives} are defined via 
$$D_\pm^\alpha f=\left((\pm\, \cdot)^\alpha \hat f(\cdot)\right)^\vee$$ with
$(\cdot)^\vee$ denoting the inverse Fourier transform and $\hat f$ denoting the Fourier transform of $f$. 
This definition for functions whose fractional derivative is in $L_1(\R)$ or $C_0(\R)$ is equivalent to the left-sided fractional derivative being defined via
$$D_+^\alpha f(x)=\frac{d^k}{dx^k}\int_{-\infty}^x\frac{(x-s)^{k-\alpha-1}}{\Gamma(k-\alpha)}f(s)\,ds; \quad \alpha<k<\alpha+1,$$ 
and the shifted Gr\"unwald approximation
$$D_+^\alpha f(x)=\lim_{h\to 0}\frac1{h^\alpha}\sum_{k=0}^\infty \begin{pmatrix}\alpha\\k\end{pmatrix}(-1)^k f(x+h-kh).$$ 
The right-sided derivative $D_-^\alpha f$ is defined via symmetry, replacing $x$ with $-x$ and $f(s)$ with $f(-s)$, giving
\begin{equation*}\begin{split}D_-^\alpha f(x)=& \frac{d^k}{(-dx)^k}\int_{-\infty}^{-x}\frac{(-x-s)^{k-\alpha-1}}{\Gamma(k-\alpha)}f(-s)\,ds\\
& \frac{d^k}{(-dx)^k}\int_{x}^{\infty}\frac{(s-x)^{k-\alpha-1}}{\Gamma(k-\alpha)}f(s)\,ds\\
=&\lim_{h\to 0}\frac1{h^\alpha}\sum_{k=0}^\infty \begin{pmatrix}\alpha\\k\end{pmatrix}(-1)^k f(x-h+kh),\quad\alpha<n<\alpha+1.\end{split}
\end{equation*}
Note that, for the forward equation $u'(t,x)=D^\alpha_+u(t,x)$, the non-local structure of the fractional derivatives encodes the rate of change at a location $x$ as a weighted average of what can jump from the left to that location and what can drift in from the nearest neighbour on the right.

\subsection{Preliminaries for fractional derivative operators on bounded domains} 
\label{sectionpreliminariesfract}

We begin with the necessary preparations in order to facilitate the definition of fractional derivative operators on a bounded domain. In what follows, to make use of the inherent symmetry of right-sided and left-sided derivatives, integrals, etc., we will work on the bounded interval $[-1,1]$. Then what holds for a left-sided operation will hold for a right-sided operation after replacing $x$ by $-x$. Therefore, we will focus on the left-sided operation unless explicitly stated.

Let $p_{\beta}$ denote the power functions given by
$$ p^\pm_{\beta}(x):= \frac{(1\pm x)^{\beta}}{\Gamma (\beta +1)},$$ $x \in (-1,1) $ and $\beta > -1$. We will usually drop the superscript if the context is clear (left-sided with $+$ and right-sided with $-$).  We use $p_0$ and $\bf{0}$ to denote the constant \textit{one} function and the \textit{zero} function on the interval $[-1,1]$, respectively.

For $\nu >0$ and $f \in L_1[-1,1]$, the (left-sided) fractional integral can be written as $$I^{\nu}_+ f(x)=\int_{-1}^x p^+_{\nu-1} (x-s-1)f(s)\;\dd s$$ and the (right-sided) fractional integral
$$I^{\nu}_- f(x)=\int_{-1}^{-x} p^+_{\nu-1} (-x-s-1)f(-s)\;\dd s=\int_x^1  p^-_{\nu-1}(x-s+1)f(s)\,ds.$$

 Note that
\begin{equation}
\label{Ialphapfunction}
I^{\nu}_\pm p^\pm_{\beta}  = p^\pm_{\beta + \nu},\; \beta>-1
\end{equation}
and that $I^{\nu}$ is a bounded linear operator on both $L_1[-1,1]$ and $C[-1,1]$. The fractional integrals satisfy the semigroup property,
\begin{equation}
\label{semigroupproperty}
I^{\nu_1}_\pm I^{\nu_2}_\pm f= I_\pm^{\nu_1 + \nu_2} f,\; f \in L_1[-1,1],\; \nu_1, \nu_2 >0.
\end{equation}
One obvious observation, which we will use later, is that for $f \in L_1[-1,1]$ such that $f$ is bounded on $[-1, -1+\epsilon)$ for some $\epsilon >0$,
\begin{equation}
\label{Inuat0}
I^{\nu}_+ f (-1) := \lim_{x\downarrow -1}I^{\nu}_+ f (x) = 0, \nu >0.
\end{equation}
Similarly, $I^\nu_-f(1)=0$ if $f$ is bounded near 1.

On the other hand, let $W^{n,1}[-1,1],\; n\in \mathbb{N}$ denote the Sobolev space of functions such that the function along with its (weak) derivatives $D^k$ up to order $n$ belong to $L_1[-1,1]$.  For $\beta > n-1,\; n \in \mathbb{N}$,
\begin{equation}
\label{integerderivativepfunction}
(\pm D)^n p^\pm_\beta = p^\pm_{\beta -n} \; \mathrm{and}\; D p_0 = \bf{0},
\end{equation}
where we highlight that for the right-sided case the formula holds if we use $-D$ instead of $D$ and $p^-_\beta(x)=(1-x)^\beta/\Gamma(\beta+1)$.

Note that for $f \in W^{1,1}[-1,1]$ and all $\nu>0$,
\begin{equation}
\label{derivativeconvolution}
D(I^{\nu}_\pm f) =  I^{\nu}_\pm (Df) + f(\mp 1)p^\pm_{\nu -1}.
\end{equation}
Hence there are two common definitions for fractional derivative operators depending on whether one fractionally integrates first and then differentiates (Riemann-Liouville) or vice-versa (Caputo). It turns out that for  fractional derivatives of order greater than one the Caputo operator with Dirichlet boundary condition is not positive and hence unsuitable for modelling stochastic processes. However, a mixed-type fractional derivative operator is still dissipative and positive.

We investigate two types of fractional derivative operators: The one-sided \emph{ mixed Caputo} and the \emph{Riemann-Liouville} fractional derivative operators of order $1<\alpha\le 2$. 
\begin{definition}
For $1<\alpha\le 2$, the \emph{mixed Caputo} and the \emph{Riemann-Liouville} fractional derivatives (of orders $\alpha$ and $\alpha -1$) are given respectively by
\begin{equation}
\label{Dalphac}
\partial_\pm^{\alpha} f:=D I_\pm^{2-\alpha}  D f , \; \partial_\pm^{\alpha-1} f:=I_\pm^{2-\alpha} Df\end{equation} and \begin{equation}\label{Dalpha}  \quad D_\pm^{\alpha} f:= D^2 I_\pm^{2-\alpha} f , \; D_\pm^{\alpha-1} f:= D I_\pm^{2-\alpha} f.
\end{equation}
\end{definition}

Note that by \eqref{derivativeconvolution}, if $f(\mp 1)=0$, then $\partial^{\alpha}_\pm f=D^\alpha_\pm f$. Furthermore,
for $1 <\alpha\le 2$, using \eqref{integerderivativepfunction}, \eqref{Ialphapfunction} and \eqref{Dalpha}, we have $\partial_\pm^{\alpha} p^\pm_\alpha = p_0 = D_\pm^\alpha p^\pm_\alpha$ and
\begin{equation}
\label{specialpDalpha}
 \partial_\pm^{\alpha} p^\pm_{\alpha -1} =\mathbf {0} = D_\pm^\alpha p^\pm_{\alpha -1}, \; \partial_\pm^{\alpha} p_0 = \mathbf{0} \; \mathrm{and}\; D_\pm^{\alpha} p^\pm_{\alpha -2}= \mathbf{0},
\end{equation} 
resulting in two degrees of freedom when inverting the fractional derivative operator. The degrees of freedom are then used up by the respective boundary conditions which are encoded in the domain of the operator.

To establish the connection between fractional differential equations with various boundary conditions and various modified stable processes we use the following theorem. 
\begin{theorem}\label{thm:fc}
A family of Feller semigroups converge strongly, uniformly for $t \in [0,t_0]$, to a Feller semigroup if and only if their respective Feller processes converge in the Skorokhod topology.
\end{theorem}
\begin{proof} See \cite[p. 331, Theorem 17.25]{Kallenberg1997}.
\end{proof}

\section{The main results}

\subsection{Boundary conditions for fractional derivative operators}
As the main contribution of this paper we identify the stochastic processes and show well-posedness of the fractional PDEs for almost any combination of the following homogeneous boundary conditions:
\begin{itemize}
\item {\bf Dirichlet:} Zero boundary condition; i.e., the function is zero at the boundary.  
\item {\bf Neumann:} The $\alpha-1$ derivative is zero at the boundary; i.e., $\partial^{\alpha-1}f$ or $D^{\alpha-1}f$ is zero at the boundary depending on whether the operator is a mixed Caputo or Riemann Liouville operator respectively. 
\item {\bf Neumann*:} The first derivative is zero at the boundary.
\end{itemize}
For $\mathrm{L},\mathrm{R}\in\{\mathrm{D},\mathrm{N},\mathrm{N^*}\}$ we refer to $\mathrm{LR}$ as the set of continuously differentiable functions on $(-1,1)$ that satisfy the respective boundary conditions; e.g., $$\mathrm{DN^*}=\{f:\lim_{x\to -1}f(x)=0,\lim_{x\to 1}f'(x)=0\},$$ with the other combinations of boundary conditions defined analogously. If we want to comprehensively talk about several operators with an unspecified left or right boundary, we just state $\mathrm{L}$ or $\mathrm{R}$; e.g., $\mathrm{LD}$ would refer to having a left boundary condition of type $\mathrm{D},\mathrm{N},$ or $\mathrm{N^*}$ and a right Dirichlet condition. 

We use the notation
$C_0(\Omega)$  to be the closure with respect to the supremum norm of the space of continuous functions with compact support in $\Omega$. The set $\Omega$ refers to the interval $$\Omega=[(-1,1)]$$ with an endpoint excluded if the problem has a Dirichlet boundary condition there (if the endpoint is excluded, any function in $C_0(\Omega)$ will converge to zero there).

We will only consider the twelve cases shown in Table  \ref{explicitdomains}. 
We will show that in those twelve cases the operators $A$ are closed, densely defined,  the range of $I-A$ is dense, and are dissipative. Therefore, they are generators of strongly continuous contraction semigroups (which also turn out to be positive) and hence the associated Cauchy problems are well-posed. 

All other cases are either given by symmetry (such as $(\partial_-^\alpha,\BC)$ for $L_1[-1,1]$), are equivalent (such as $(\partial_+^\alpha,\mathrm{DR})$ and $(D_+^\alpha,\mathrm{DR})$, or  $(\partial_+^\alpha,\mathrm{NR})$ and the closure of $(\partial_+^\alpha,\mathrm{N^*R})$), or do not generate a contraction semigroup on the respective space.
We will not pursue these other cases any further in this manuscript. 

%
%
\begin{table}
\centering
\vline
\begin{tabular}{l|l|}
  \hline
  \multicolumn{2}{c}{$\mathcal D(A,\BC)=\left\{f\in X\cap \BC:f=I^\alpha g+ap_\alpha+bp_{\alpha-1}+cp_\eta,g\in X, g+ap_0\in X\right\}$}\vline\\ \hline

  \multicolumn{2}{c}{$X = L_1[-1,1]$}  
	\vline \\
	\hline
1.	 & $\D (\partial^\alpha_+, \mathrm{DD}) = \left\{f\in X:\;f=I^\alpha_+ g- \frac{I^\alpha_+ g (1)}{p^+_{\alpha-1}(1)}p^+_{\alpha-1},\;g\in X\right\}$\\
  \hline
2.	 & $\D (\partial^\alpha_+, \mathrm{DN})=\left\{f\in X:\;f=I^\alpha_+ g- [I_+^1 g (1)]p^+_{\alpha-1},\;g\in X\right\}$\\
  \hline
3.	& $\D (\partial^\alpha_+, \mathrm{ND})=\left\{f\in X:\;f=I_+^\alpha g -[I_+^\alpha g(1)]p_0,\;g\in X\right\}$\\
  \hline
4.	 & $\D (\partial^\alpha_+, \mathrm{NN})=\left\{f\in X:\;f=I_+^\alpha g-\frac{I_+^1g(1)}{p_1^+(1)} p^+_{\alpha}+	cp_0, \;g\in X , c\in\R\right\}$\\
  \hline
5.	 & $\D (D^\alpha_+, \mathrm{ND})=\left\{f\in X:\;f=I^\alpha_+ g - \frac {I^\alpha_+ g(1)}{p^+_{\alpha-1}(1)} p^+_{\alpha -2},	\;g\in X \right\}$\\
  \hline
6.	 & $\D (D^\alpha_+, \mathrm{NN})=\left\{f\in X:\;f=I^\alpha_+ g-\frac{I_+^1g(1)}{p^+_1(1)} p^+_\alpha+	cp^+_{\alpha -2},	\;g\in X , c\in\R\right\}$\\
  \hline
  \multicolumn{2}{c}{$X = C_0(\Omega)$}  
	\vline \\
	\hline
1.	 & $\D (\partial^\alpha_-, \mathrm{DD}) = \left\{f\in X:\;f=I^\alpha_- g-\frac{I^\alpha_- g (-1)}{p^-_{\alpha-1}(-1)}p^-_{\alpha-1},\;g\in X\right\}$\\
  \hline
2.	 & $\D (\partial^\alpha_-, \mathrm{DN})=\left\{f\in X:\;f=I_-^\alpha g- [I^\alpha_- g (-1)]p_0,\;g\in X\right\}$\\
  \hline
3.	 & $\D (\partial^\alpha_-, \mathrm{ND})=\left\{f\in X:\;f=I_-^\alpha g -[I^1_- g(-1)]p^-_{\alpha -1}, \;g\in X\right\}$\\
  \hline
4.	 & $\D (\partial^\alpha_-, \mathrm{NN})=\left\{f\in X:\;f=I_-^\alpha g-\frac{I^1_-g(-1)}{p^-_1(-1)}p^-_\alpha +c p_0, \;g\in X, c\in\R  \right\}$\\
  \hline
5.	 & $\D (\partial^\alpha_-, \mathrm{N^*D})=\left\{f\in X:\;f=I^\alpha_- g- \frac{I_-^{\alpha-1} g(-1)}{p_{\alpha-2}^-(-1)} p^-_{\alpha -1},	\;g\in X \right\}$\\
  \hline
6.	& $\D (\partial^\alpha_-, \mathrm{N^*N})=\left\{f\in X:\;f=I^\alpha_- g-\frac{I_-^{\alpha-1}g(-1)}{p^-_{\alpha-1}(-1)} p^-_\alpha+	cp_0,	\;g\in X, c\in\R \right\}$\\
  \hline
\end{tabular}
\caption{\label{explicitdomains}
Explicit domains of left-sided fractional derivative operators $(A^+,\BC)$ on $L_1[-1,1]$ and right-sided $(A^-,\BC)$ on $C_0(\Omega)$. For $f\in \D(A,\BC), Af=g+a$, where $a$ is the coefficient of the $p_\alpha$ term. The domains consist precisely of those functions that satisfy the boundary conditions and are fractionally differentiable in $X$ (first line of table, see Proposition \ref{equivalentdefinitions}) and correspond to the domains of the forward and backward generators of the six processes in Table \ref{explicitProcesses}.}
\end{table}  

\subsection{Connection to stochastic proceses}
For the cases in Table \ref{explicitdomains} we are going to prove the well-posedness of the Feller equation on $C_0(\Omega) $ and the Fokker-Planck equation on $L_1[-1,1]$. Furthermore, we identify the processes $Z_t$ for which
$$P(t)f(x)={\mathbb E}^x[f(Z_t)]$$
and
$$P^*(t)f=\mbox{``Probability density function of }Z_t|Z_0\sim f\mbox{''}$$
 are solutions to the respective differential equations 
 $$u'(t)=Au(t);\;\;\; u(0)=f\in X.$$ Here $\mathbb E^x$ is the expectation conditioned on $Z_0=x$, $Z_0\sim f$ means that $Z_0$ has probability density function $f$, and $A$ is the fractional derivative operator of Table \ref{explicitdomains}.

This is achieved by approximating the generators as well as the processes with `finite state' approximations. These approximations will show that the operators $(A,\BC)$ are dissipative which is a key requirement when using the Lumer-Phillips Theorem to show that the corresponding Cauchy problems are well-posed. The Trotter-Kato Theorem together with Theorem \ref{thm:fc} then yields the characterizations of the associated stochastic processes, as they are approximated in the Skorokhod topology by their finite state counterparts. 

The processes, based on the spectrally positive $\alpha$-stable process $Y_t$,  are given in Table \ref{explicitProcesses}, where we denote for a process $V_t$,
$$ (V_t)^{\mathrm{kill}}=\begin{cases} V_t & t<\inf\{\tau:V_\tau\not\in [-1,1] \},\\ \delta& \mathrm{else,}\end{cases}$$
where $\delta$ is a grave yard point, and $$V_t^{\max}=\sup_{0\le\tau\le t}\{V_\tau,1\}\mbox{ and } V_t^{\min}=\inf_{0\le\tau\le t}\{V_\tau,-1\}.$$  The time forwarded processes  $V_{E_t^l}$ are defined using the processes
$$E_t^l=\max\left\{\tau: \int_0^\tau 1_{\{V_s\ge -1\}}(s)\,ds=t\right\},E_t^r=\max\left\{\tau: \int_0^\tau 1_{\{V_s\le 1\}}(s)\,ds=t\right\}$$ and
$$E_t^{l,r}=\max\left\{\tau: \int_0^\tau 1_{\{-1\le V_s\le 1\}}(s)\,ds=t\right\}.$$

We give the following interpretation of the six processes in Table \ref{explicitProcesses}:
\begin{enumerate}
\item The  process is killed as soon as the process leaves $[-1,1]$. 
\item The  process is killed if it drifts across the left boundary. If it jumps across the right boundary we make a time change and delete the time for which $Y_t$ is to the right of the right boundary (in other words, fast-forward to the time the process is again to the left of the right boundary each time it jumps across the right boundary).
\item Here we make a time change and delete the time for which $Y_t$ is to the left of the left boundary (in other words, fast-forward to the time the process is to the right of the left boundary each time it drifts across the left boundary).  This process is killed if it jumps across the right boundary. 
\item Here we make a time change and delete the time for which $Y_t$ is to the left of the left boundary or the right of the right boundary. 
\item This process is killed if it jumps across the right boundary and reflected at the left boundary.
\item This process is reflected at the left boundary and fast forwarded on the right boundary.
\end{enumerate}
Note that for a jump across the right boundary this time change (which was considered on the halfline in e.g. \cite{Bertoin1992,Patie2017}) is equivalent to restarting the process near the right boundary (stochastically reflecting) due to the fact that the infimum process is continuous. However, for a drift across the left boundary this time change is different from reflecting the process on the left boundary as the maximum process is a pure jump process and therefore the time forwarded process will restart inside the domain. 
\begin{table}
\centering
\vline
\begin{tabular}{l|c|c}
  \hline
\hspace{1cm} Process $Z_t$ &
\hspace{.1cm}	Forward generator 
\hspace{.1cm} & \hspace{.1cm} Backward generator
\hspace{.1cm}\\
	\hline
	1.  $Z_t=Y^{\mathrm{kill}}_t$& $(\partial_+^\alpha, \mathrm{DD})$  & $(\partial_-^\alpha, \mathrm{DD})$   \\
  \hline
2.	$Z_t=\left(Y_{E^r_t}\right)^{\mathrm{kill}}$ & $(\partial_+^\alpha, \mathrm{DN})$ & $(\partial_-^\alpha, \mathrm{DN})$\\
  \hline
3.	 $Z_t=\left(Y_{E^l_t}\right)^{\mathrm{kill}}$& $(\partial_+^\alpha, \mathrm{ND})$ &$(\partial_-^\alpha, \mathrm{ND})$\\
  \hline
4.	$Z_t=Y_{E^{l,r}_t}$ & $(\partial_+^\alpha, \mathrm{NN})$ & $(\partial_-^\alpha, \mathrm{NN})$\\
  \hline
5.	$Z_t=(Y_t-Y^{\min}_t-1)^{\mathrm{kill}}$ &$(D_+^\alpha, \mathrm{ND})$ & $(\partial_-^\alpha, \mathrm{N^*D})$\\
  \hline
6.	 $Z_t=(Y_t-Y^{\min}_t-1)_{E^r_t}$ & $(D^\alpha_+, \mathrm{NN})$ & $(\partial_-^\alpha, \mathrm{N^*N})$ \\
  \hline 
\end{tabular}\vline
\caption{\label{explicitProcesses}
Stochastic processes related to the spectrally positive $\alpha$-stable process $Y_t$ and the generators of the forward and backwards semigroups $(P^*(t))_{t\ge 0}$ on $L_1[-1,1]$ and $(P(t))_{t\ge 0}$ on $C_0(\Omega)$ respectively.}
\end{table}

\section{ Fractional derivative operators with boundary conditions}

In this section we show that the fractional derivative operators $A$ of Table \ref{explicitdomains} are densely defined, closed, and that $ I-A$ has a dense range for each of the fractional derivative operators $A$. Furthermore we give a core of the domains consisting of fractional polynomials satisfying the boundary conditions.


\subsection{The domain of fractional derivative operators}
It is the domain of the operator that captures the boundary conditions. In order to refer to specific operators we denote with $(A,\mathrm{LR})$ all of the fractional derivative operators (either mixed Caputo or Riemann-Liouville) that have left and right boundary conditions and replace $A,\mathrm{L},\mathrm{R}$ if we need to be more specific. For example,
$(\partial^\alpha_-,\mathrm{N^*R})$ refers to all of the mixed Caputo operators with a Neumann* left boundary condition and any right boundary condition. 

Recall that  $\mathrm{LR}$ refers to the set  continuously differentiable functions on $(-1,1)$ that satisfy the respective boundary conditions.
\begin{definition}
\label{definitionfractionalderivativeoperatoronX}
The pair $(A^\pm , \BC)$ for any combination of $\mathrm{L},\mathrm{R}\in\{\mathrm{D},\mathrm{N},\mathrm{N^*}\}$ boundary conditions is called a fractional derivative operator on $$X\in\{L_1[-1,1], C_0(\Omega)\},$$ if $A^\pm \in \left\{\partial_\pm^{\alpha}, D_\pm^{\alpha} \right\}$, $1< \alpha \le 2$, with domain of the form
\begin{equation}\label{domaincandidatefractionalderivativeoperatoronX}
\D(A^\pm, \BC) = \{ f \in X\cap \BC\;: \; f = I_\pm^{\alpha} g +a p^\pm_\alpha+ b p^\pm_{\alpha-1} + c p^\pm_{\eta},g\in X, g+ap_0\in X \}
\end{equation}
with  $\eta=0$ if $A^\pm=\partial^\alpha_\pm$, and $\eta=\alpha-2$ if $A^\pm=D_\pm^\alpha$. Note that $A^\pm f = g + ap_0$ for all $f$ in the respective domain.

\end{definition}


For the twelve cases of Table \ref{explicitdomains} the boundary conditions completely specify the coefficients $a$, $b$, and $c$.
 
\begin{proposition}
\label{equivalentdefinitions}
Consider the fractional derivative operators $(A, \BC)$ of the first column of Table \ref{explicitdomains}. The domains given in Definition \ref{definitionfractionalderivativeoperatoronX} are equal to the domains given in Table \ref{explicitdomains}.  
\end{proposition}

\begin{proof}
A simple calculation reveals that for each of the operators given in Table \ref{explicitdomains}, the domains $\D (A, \BC)$ are a subset of the domains given by \eqref{domaincandidatefractionalderivativeoperatoronX}. For example, for $f\in \D(\partial_-^\alpha,\mathrm{ND})$; i.e. $f=I^\alpha_- g-[I^1_-g](-1)p_{\alpha-1}^-$ for some $g\in C_0([-1,1))$, we have to show that $f\in \mathrm{ND}$. As $f(1)=0$ and $$\partial^{\alpha-1}_-f(-1)= [\partial^{\alpha-1}_-I^\alpha_- g-[I^1_-g](-1)p_{\alpha-1}^-](-1)=I^1_-g(-1)-I^1_-g(-1)=0,$$ that is satisfied. Similarly, $f$ satisfies the boundary conditions in all other cases.

On the other hand, for $f$ given in Definition \ref{definitionfractionalderivativeoperatoronX}; that is, $f = I^{\alpha} g +a p_\alpha+ b p_{\alpha-1} + c p_{\eta}, \; g \in X$ satisfying $\BC$ such that $Af \in X$, we need to show that $f$ belongs to the corresponding domains given explicitly in Table \ref{explicitdomains}. 

First, we are going to show that if there is a Dirichlet boundary condition, then we may specify  \eqref{domaincandidatefractionalderivativeoperatoronX} more explicitly  by setting $a=0$. If $X = L_1[-1,1]$, since $a p_\alpha = I^\alpha (a p_0)$, the term $ap_0$ can just be incorporated into $g \in L_1[-1,1]$ and thus we may set $a =0$ without loss of generality.  On the other hand, if $X=C_0(\Omega)$, then $g\in C_0(\Omega)$ and $A f =g+ap_0 \in C_0(\Omega)$; in particular, if $g$ has to be zero at one of the endpoints due to the Dirichlet condition so does $Af$, and hence $a=0$.    

For $(A^+,\mathrm{DR})$ and $(\partial_-^\alpha,\mathrm{LD})$, we show that this boundary condition implies that $c=0$ in  \eqref{domaincandidatefractionalderivativeoperatoronX}. We have shown that $a=0$ and in case of $\lim_{x \to \mp1}f(x) =0$  we have $0= I^\alpha_\pm g (\mp1)+b p_{\alpha-1}^\pm(\mp1)+ c p_\eta^\pm(\mp1)$. Since $p_{\alpha-1}^\pm(\mp1)=0$ and $I^ \alpha_\pm g (\mp1)=0$, this implies that $c=0$. 

For $(A^+,\mathrm{NR})$ and $(\partial_-^\alpha,\mathrm{LN})$, we show that $b=0$ in  \eqref{domaincandidatefractionalderivativeoperatoronX}. The Neumann boundary condition implies that  $0=I^\pm g(\mp1)+ap_1^\pm(\mp1) + b$. Thus, since the first two terms are zero, it follows that $b=0$. 

For $(A^+,\mathrm{LD})$ and $(\partial_-^\alpha,\mathrm{DR})$, since $\lim_{x \to \pm1} f(x) =0$ and $a=0$, we have $0= I_\pm^\alpha g (\pm1)+b p_{\alpha-1}(\pm1)+ c p_{\eta}(\pm1)$ or $bp^\pm_\alpha(\pm1)+cp^\pm_\eta(\pm1) = - I_\pm^\alpha g (\pm1)$.

For $(A^+, \mathrm{LN})$ and $(\partial_-^\alpha, \mathrm{NR})$,  observe that the Neumann boundary condition implies that  $I^1_\pm g(\pm 1) + a p_1^\pm(\pm1) + b=0$ or $ap_1^\pm(\pm1) + b = - I^1_\pm g(1)$.  

For $(\partial_-^\alpha, \mathrm{N^*R})$, since $Df(-1)=0$ and $ -D f = I_-^{\alpha -1}g+ a p^-_{\alpha -1}+ b p^-_{\alpha -2}$, we have 
$ap^-_{\alpha-1}(-1)+bp^-_{\alpha-2}(-1) = - I_-^{\alpha -1}g (-1)$. 

Hence $a,b$, and $c$ are specified according to Table \ref{explicitdomains} by considering both the boundary conditions simultaneously for each $\BC$.
\end{proof}

\begin{theorem}
\label{theoremAdenseA}
The fractional derivative operators $(A ,\BC )$ given by Proposition \ref{equivalentdefinitions} are densely defined.
\end{theorem}

\begin{proof}
Given $\epsilon >0$ and $\phi \in X$, for each $(A, \BC)$ we need to show that there exists $f_\epsilon \in \D(A, \BC )$ such that $\left\|\phi - f_\epsilon\right\|_X < \epsilon$. 
We are first going to show that we can approximate every $\phi\in C_0^\infty(-1,1)$ function by elements in the domain.

Let $0<\epsilon<1$ and define \begin{equation*}
h_\epsilon(x) = \begin{cases}
0, &  -1\leq x \leq 1-\epsilon,\\
\epsilon p_{\alpha+1}^+(x-2+\epsilon)- (\alpha+2)p_{\alpha+2}^+(x-2+\epsilon), & 1- \epsilon < x \leq 1
\end{cases}
\end{equation*} 
and let $h_\epsilon^\pm(x)=h_\epsilon(\pm x)$. Consider
$$f^\pm_\epsilon=\phi+C(\epsilon)h_\epsilon^\pm$$ with '$+$' if $X=L_1[-1,1]$ and '$-$' if $X=C_0(\Omega)$ and $C(\epsilon)$ yet to be determined. As $h_\epsilon(1)=0$, $f^\pm_\epsilon\in X$ for any $C(\epsilon)$.  We have to show that for each case there exist $C(\epsilon)$ such that $f^\pm_\epsilon$ is in the domain and that $C(\epsilon)h_\epsilon^\pm\to0$. 

Observe that $f_\epsilon^+$  satisfies any left boundary condition and $f_\epsilon^-$ any right boundary condition as $f_\epsilon^\pm$ is zero in a neighbourhood of the respective boundary. Furthermore, as $f_\epsilon(-1)=0$, the mixed Caputo derivative is equal to the Riemann-Liouville derivative and
$$\partial^\alpha h_\epsilon(1)=-\frac{\alpha}{2}\epsilon^2, \partial^{\alpha-1}h_\epsilon(1)=-\frac{\alpha-1}{6}\epsilon^3, Dh_\epsilon(1)=-\frac1{\Gamma(\alpha+2)}\epsilon^{\alpha+1}.$$
This allows us to set the relevant boundary value of $f_\epsilon$ to zero. For example, to satisfy a right Dirichlet condition for the mixed Caputo operator we set $$C(\epsilon)=-\frac{\partial^\alpha_+\phi(1)}{-\frac{\alpha}{2}\epsilon^2}.$$ Similarly in all the other cases we can find $C(\epsilon)$ such that $f_\epsilon^\pm\in \D(A,\BC)$ and $C(\epsilon)=O(\epsilon^{-3})$. As  $|h_\epsilon(x) |\le \epsilon^{\alpha+2}/\Gamma(\alpha+2)$ for all $x$ we see that in all cases $C(\epsilon)h_\epsilon\to 0$ and hence $f_\epsilon^\pm\to\phi$. 

As $C_0^\infty  (-1,1)$ is dense in $L_1[-1,1]$, the domains of fractional derivative operators are dense in $L_1[-1,1]$. For $f\in C_0(\Omega)$, consider 
$$g=f-\frac{f(-1)-f(1)}{p^-_{\alpha+1}(-1)}p^-_{\alpha+1}-f(1)p_0\in C_0(-1,1).$$
Let $\phi\in C_0^\infty(-1,1)$ be close to $g$.
Similarly to above we can find $C(\epsilon)$ such that $$f_\epsilon=\phi+\frac{f(-1)-f(1)}{p^-_{\alpha+1}(-1)}p^-_{\alpha+1}+f(1)p_0+C(\epsilon)h_\epsilon^-\in \D(A^-,\BC)$$  is close to $f$ and hence the domains of the fractional derivative operators are dense in $C_0(\Omega)$ as well.
\end{proof}

\subsection{The closedness of fractional derivative operators}

\begin{proposition}
\label{propositionAinvertible}
The operators $(A , \BC)$ with at least one Dirichlet boundary condition; i.e., $\BC \in \{ \mathrm{DR}, \; \mathrm{LD}\}$, are (boundedly) invertible.
\end{proposition}
\begin{proof}
For each $(A , \BC)$ on $X$ with $\BC \in \{ \mathrm{DR}, \; \mathrm{LD}\}$, we show that there is a bounded operator $B$ on $X$ such that $BAf = f$ for all $f \in \D(A, \BC)$, and for all $ g \in X$, $Bg \in \D(A, \BC)$ and $ABg=g$. 
Consider the operators $B: X \to X$ given by $B g = I^\alpha g + b p_{\alpha -1} + c p_\eta,\; g \in X$
with $\eta$ as in Definition \ref{definitionfractionalderivativeoperatoronX}. For each $(A, \BC)$  we take the coefficients $b$ and $c$ from Table \ref{explicitdomains}; they depend continuously on $g$ in $X$. Then, it is clear that $B g \in \D(A, \BC)$ in view of Proposition \ref{equivalentdefinitions}, and since $I^\alpha$ is a bounded linear operator on $X$, so is $B$. 

By definition,  $A B g =A \left(I^\alpha g + b p_{\alpha -1} + c p_\eta \right) = g$.
To complete the proof, using \eqref{derivativeconvolution} and \eqref{specialpDalpha}, we know that 
$$B \partial_\pm^\alpha f = f +\left(b-\partial_\pm^{\alpha -1}f(\mp1)\right) p^\pm_{\alpha -1}+ \left(c-f(\mp1)\right) p_0$$ and 
$$B D_+^\alpha f = f + \left(b - D_+^{\alpha -1}f(-1)\right) p^+_{\alpha -1}+ \left(c - I_+^{2-\alpha}f(-1)\right)p^+_{\alpha -2}.$$
Hence, using the conditions on $b,c$ from Table \ref{explicitdomains}, we see that $B A f =f$ for all $f \in \mathcal{D}(A, \BC)$ with $\BC \in \{\mathrm{DR},\; \mathrm{LD}\}$, and therefore 
the operators $(A, \BC)$ with at least one Dirichlet boundary condition are invertible.
 \end{proof}

\begin{theorem}
\label{theoremClosedA}
The fractional derivative operators $(A ,\BC )$ given by Proposition \ref{equivalentdefinitions} are  closed  operators on $X$.
\end{theorem}
\begin{proof}
By Proposition \ref{propositionAinvertible}, if the domains encode at least one Dirichlet boundary condition, then the operators $(A , \BC)$ are invertible and hence closed. Consider the remaining four cases, namely, 
\begin{equation}
\label{NNfourcases}
 \begin{split}(\partial_\pm^\alpha, \mathrm{NN})& \mbox{ on } L_1[-1,1] \mbox{ and } C_0[-1,1]\;\mbox{ respectively,}\\
  (D_+^\alpha, \mathrm{NN})& \mbox{  on }  L_1[-1,1] \\ (\partial_-^\alpha, \mathrm{N^*N})&  \mbox{ on }C_0[-1,1].\end{split}
\end{equation}
For each $(A, \BC)$ on $X$ given in \eqref{NNfourcases}, we show that if $\left\{f_n \right\} \subset \D(A, \BC)$ such that $f_n\to f$ and $Af_n\to \phi$ in $X$, then $f\in \D(A, \BC)$ and $Af=\phi$.  Using Table \ref{explicitdomains}, consider the sequence $\left\{f_n \right\} \in \D(A, \BC)$ given by $ f_n=I^\alpha g_n+a_n p_\alpha + c_n p_\eta,\; g_n \in X$, 
where $\eta$ is given by Definition \ref{definitionfractionalderivativeoperatoronX}. Note that for theses four cases either  $a_n=-I_\pm g_n(\pm1)/p^\pm_1(\pm1)$ or $a_n= -  I_-^{\alpha -1}g_n(-1)/p^-_{\alpha-1}(-1)$, and $A f_n = g_n + a_n p_0$.  Let $f_n \to f$ and $A f_n = g_n +a_n p_0  \to \phi$ in $X$. Then, since $I^\alpha$ is bounded, 
$I^\alpha\left(g_n+a_np_0\right)=I^\alpha g_n+a_n p_\alpha \to I^\alpha \phi$. This implies that $ c_n p_\eta = f_n - (I^\alpha g_n+a_n p_\alpha) \to f - I^\alpha \phi \in X$ and thus, there exists $c$ such that $c_n \to c$. Hence, $ f = I^\alpha \phi + c p_{\eta}$. 

As $g_n+a_np_0\to\phi$,  in the  case of $a_n=-Ig_n(1)/p_1(1)$ we have  that $I_\pm\phi(\pm1)=0$, and in the case of $a_n= -  I_-^{\alpha -1}g_n(-1)/p^-_{\alpha-1}(-1)$ we obtain  $I_-^{\alpha-1}\phi(-1)=0$. By Proposition \ref{equivalentdefinitions}, $f\in\D(\BC)$  and hence  the operators $(A, \BC)$ given in Definition \ref{definitionfractionalderivativeoperatoronX} are closed in $X$.
\end{proof}

\subsection{A core for fractional derivative operators}
 
In the following, we refer to  
\begin{equation}\label{remarkpolynomialsinX} P  = \sum_{m=0}^{N} k_m p_m
\end{equation} as polynomials (with integer powers) belonging to $X$ where the constants $k_m$ are constrained to ensure that $P \in X$; i.e., they might have to be constrained so that $P$ satisfies Dirichlet boundary conditions. We will repeatedly make use of the Stone-Weierstrass theorem, e.g. \cite[Corollary 4.50]{Folland1999}, stating that the polynomials belonging to $X$ are dense in $X$.
\begin{theorem}
\label{theoremCore} The subspace 
\begin{equation}
\label{coreA}
\mathcal C(A , \BC )=\left\{f : f = I^\alpha P +a p_\alpha+b p_{\alpha-1}+c p_{\eta},P\in X\mbox{ polynomial} \right\}
\end{equation}
is a core of the fractional derivative operators $(A ,\BC )$ given by Proposition \ref{equivalentdefinitions} if for each polynomial $P \in X$,  $\eta$ is given in Definition \ref{definitionfractionalderivativeoperatoronX} and the coefficients $a,\;b,\; c\in\R$ are given by Table \ref{explicitdomains}.
\end{theorem}
\begin{proof}
We need to show that for $f \in \D (A, \BC)$ there exists $f_n \subset \mathcal C(A, \BC)$ such that $f_n\to f$ and $Af_n\to Af$ in $X$. 

Let $f = I^\alpha g + a p_\alpha + bp_{\alpha -1}+ c p_\eta\in \D(A,\BC)$ and let $P_n\to g$ with $P_n =\sum_{m=0}^{N_n} k_m p_m \in X$.  Consider $f_n=I^\alpha P_n + a_n p_\alpha + b_n p_{\alpha-1}+c_n p_\eta$ as given in Table \ref{explicitdomains}. Note that
 $a_n, b_n, c_n$ depend continuously on $P_n$ in $X$. If we have at least one Dirichlet boundary condition, then observe that $a_n =0$ while $b_n, c_n$ are either zero or depend continuously on $P_n$ in $X$. On the other hand, if there are no Dirichlet boundary conditions then $b_n=0$ while $a_n$ depends continuously on $P_n$ in $X$ and since $c$ is free, we set $c_n =c$. For $\nu >0$, $I^\nu$  is bounded, and thus continuous on $X$, thus $I^\nu P_n \to I^\nu g$ in $X$ for $\nu \in \left\{ \alpha, \; 1, \; \alpha -1 \right\}$. Hence, $f_n \to f$ in $X$. Moreover, using \eqref{specialpDalpha} we have that $Af_n=P_n + a_n p_0$ and $Af = g + ap_0$. Thus, $Af_n \to Af$ in $X$ and hence, $\mathcal{C}(A, \BC)$ is a core of $(A, \BC)$. 
\qquad
\end{proof}

\subsection{The Range of Fractional Derivative Operators}
We now show that for each $(A, \BC)$, $\mathrm{rg}( I - A)$ is dense in $X$ by showing that for each polynomial $P\in X$ we can construct a function $\varphi \in \D(A,\BC)$ such that  $( I - A) \varphi = P$.  To this end, let $H_{\alpha , \beta} (x)= (1+x)^\beta E_{\alpha, \beta}((1+x)^\alpha),\; \alpha >0, \;\beta >-1$, where $E_{\alpha, \beta}$ denotes the standard two parameter Mittag-Leffler function \cite{Haubold2011}; that is, 
$$ H_{\alpha, \beta} (x) = \sum_{n=0}^\infty p_{n\alpha+\beta}(x), \textnormal {for}  \; x \in [-1,1], \; \alpha >0, \;\beta >-1.$$ Note that $H_{\alpha , \beta}  \in L_1[-1,1]\; \mathrm{and\; if \;} \beta \geq 0, \; H_{\alpha , \beta} \in C[-1,1]$.  Moreover, since  $I^\nu$ is bounded on $L_1[-1,1]$ and the first derivative operator $\left(D , W^{1 ,1}[-1,1]\right)$ is closed in $L_1[-1,1]$,
\begin{equation}
\label{RecurrenceInuandDmittag}\begin{split}
 H_{\alpha,\beta}=& p_{\beta}+ H_{\alpha,  \beta + \alpha}  \\
 I^{\nu} \left(H_{\alpha , \beta} \right) =& I^{\nu}\left(\sum_{n=0}^\infty p_{n\alpha+\beta} \right)  = \sum_{n=0}^\infty p_{n\alpha+\beta+\nu} = H_{\alpha , \beta+\nu}, \; \nu >0, \; \beta>-1\\
 D H_{\alpha, \beta} =&  \sum_{n=0}^\infty D p_{n \alpha + \beta} =\sum_{n=0}^\infty p_{n \alpha + \beta-1} = H_{\alpha, \beta -1}, \; \beta >0.
\end{split}\end{equation}

Using \eqref{specialpDalpha}  and \eqref{RecurrenceInuandDmittag} along with $\eta$ as in Definition \ref{definitionfractionalderivativeoperatoronX} we obtain for $m \in \mathbb{N}_0$, 
\begin{align}
\label{derivativesofMittag}
 D H_{\alpha, 0} =& H_{\alpha, \alpha -1}, &  D H_{\alpha, \alpha-1} =& H_{\alpha , \alpha -2},\nonumber\\ 
 \partial^{\alpha-1} H_{\alpha, \alpha-1} =& H_{\alpha , 0}, & \partial^{\alpha-1} H_{\alpha, 0} =& H_{\alpha , 1} ,\nonumber \\
  D^{\alpha -1} H_{\alpha, \alpha -2} =& H_{\alpha, \alpha -1}, & D^{\alpha-1} H_{\alpha, \alpha +m}  =&\partial^{\alpha-1} H_{\alpha, \alpha +m}= H_{\alpha, m +1}, \nonumber \\
  A H_{\alpha, \alpha-1} =&H_{\alpha, \alpha-1}, & A H_{\alpha,\eta} = &H_{\alpha,\eta}, \quad A H_{\alpha , \alpha +m} = p_m + H_{\alpha, \alpha +m}.
\end{align}
Let $H^\pm_{\alpha,\beta}(x)=H_{\alpha,\beta}(\pm x)$.

\begin{theorem}
\label{denseRangeIminusA}
Let $(A, \BC)$ denote the fractional derivative operators on $X$ as in Definition \ref{definitionfractionalderivativeoperatoronX}. Then, $\mathrm{rg}(I-A)$ are dense in $X$ for each $(A, \BC)$.
\end{theorem}
\begin{proof}
Note that the polynomials  are dense in $X$. For each $P = \sum_{m =0}^N k_m p_m^\pm \in X $ and $(A, \BC)$, where we take '$+$' if $X=L_1[-1,1]$ and '$-$' if $X+C_0(\Omega)$, let 
$$ \varphi = - \sum_{m =0}^N k_m H^\pm_{\alpha , \alpha +m} + r H^\pm_{\alpha, \alpha-1} + s H^\pm_{\alpha , \eta},$$
with $r$ and $s$ as in Table \ref{tableconstantsmittag} and $\eta$ given by Definition \ref{definitionfractionalderivativeoperatoronX}. To show that $\mathrm{rg}(I-A)$ is dense in $X$, we are gong to show that $\varphi \in \D (A, \BC)$ and $(I-A) \varphi = P$. 

 Firstly, using \eqref{RecurrenceInuandDmittag} observe that for each $(A, \BC)$, $\varphi$ is of the required form $$\varphi=I_\pm^\alpha g -k_0 p^\pm_\alpha+ r p^\pm_{\alpha-1}+ s p^\pm_{\eta},$$ where $g = -\sum_{m =1}^N k_m H^\pm_{\alpha , m}-k_0 H^\pm_{\alpha, \alpha} +r H^\pm_{\alpha, \alpha -1}+s H^\pm_{\alpha, \eta}$ and $A \varphi = g - k_0 p_0$. Secondly, it is straightforward, using \eqref{derivativesofMittag}, to verify that $\varphi \in \D(A, \BC)$.
 To complete the proof note that, for each $(A, \BC)$,  it follows by \eqref{derivativesofMittag} that $$(I-A)\varphi = \sum_{m=0}^N k_m p_m^\pm =P.$$ Hence, $\mathrm{rg}(I-A)$ is dense in $X$ for each $(A, \BC)$. 
\qquad \end{proof}

\begin{table}
\centering
\vline
\begin{tabular}{c|c|c|c|c}
   \Xhline{4\arrayrulewidth}
  $L_1[-1,1]$ &$C_0(\Omega)$ &r & s \\ 
    \Xhline{4\arrayrulewidth}
  $(A^+, \mathrm{DD})$ &$(\partial^\alpha_-, \mathrm{DD})$ &$\frac{ \sum_{m =0}^N k_m H^\pm_{\alpha , \alpha +m}(\pm1)} {H_{\alpha , \alpha -1 }(\pm1)}$  & 0\\
  \hline
  $(A^+, \mathrm{DN})$ &$(\partial^\alpha_-, \mathrm{ND})$  &$\frac{ \sum_{m =0}^N k_m H^\pm_{\alpha , m+1}(\pm1)} {H_{\alpha , 0 }(\pm1)}$ &0 \\
  \hline
  &$(\partial^\alpha_-, \mathrm{N^*D})$ &$\frac{ \sum_{m =0}^N k_m H^-_{\alpha , \alpha +m-1}(-1)} {H^-_{\alpha , \alpha -2 }(-1)}$&0\\
  \hline
  $(A^+, \mathrm{ND})$ &$(\partial^\alpha_-, \mathrm{DN})$ & 0 & $\frac{ \sum_{m =0}^N k_m H^\pm_{\alpha , \alpha +m}(\pm1)} {H^\pm_{\alpha , \eta }(\pm1)}$   \\
  \hline
  $(A^+, \mathrm{NN})$ &$(\partial^\alpha_-, \mathrm{NN})$  & 0 & $\frac{ \sum_{m =0}^N k_m H^\pm_{\alpha , m+1}(\pm1)} {H^\pm_{\alpha , \eta +1 }(\pm1)}$ \\ 
  \hline
  &$(\partial^\alpha_-, \mathrm{N^*N})$ & 0 & $\frac{ \sum_{m =0}^N k_m H^-_{\alpha , \alpha +m-1}(-1)} {H^-_{\alpha , \alpha -1 }(-1)}$ \\ 
  \hline
\end{tabular}
\caption{\label{tableconstantsmittag}
Constants for $\varphi - \sum_{m =0}^N k_m H^\pm_{\alpha , \alpha +m} + r H^\pm_{\alpha, \alpha-1} + s H^\pm_{\alpha , \eta}$ used in the proof of Theorem \ref{denseRangeIminusA}.}
\end{table}

\section{Extension of a finite state Markov process to a Feller process}
\label{sectionmarkov}

In this section we develop a new method to infer properties of a Feller process by spatially discrete approximations. 
To prove the main results in Section 2 we show that the operators in Section 3 can be approximated using dissipative finite volume schemes. These finite volume scheme operators can be identified with generators of finite state Markov processes whose behaviour at the boundary points are easily identified.   
To exploit the fact that convergence, uniformly for $t \in [0,t_0]$, of Feller semigroups on $C_0(\Omega)$ implies convergence of the processes we show how to turn a $(n)$-state (sub)-Markov process $(X^{n}_t)_{t \geq 0}\in \{1,\;2,\ldots \; , n\}$ into a Feller process $(\tilde X^{n}_t)_{t \geq 0}\in \Omega$ by having parallel copies of the finite state processes whose transition matrices interpolate continuously. The main idea here is to divide the interval $[-1,1]$ into $n+1$ grids of equal length $h$ so that the (Feller) process can jump between grids only in multiples of $h$. The transition rates for the (Feller) process $(\tilde X^{n}_t)_{t \geq 0}$ in the interval $[-1+(i-1)h,-1+ih]$ jumping up or down by $jh$ interpolate continuously between the transition rates of the finite state sub-Markov process $( X^{n}_t)_{t \geq 0}$ being in state $i-1$ going to state $(i-1+j)$ and the transition rates of being in state $i$ going to state $(i+j)$. 

In order to properly describe the interpolated transition rates and the resulting generator we start with some notation.

\begin{definition}
\label{definitiongridcoordinatefunctions}
Let $n \in \mathbb{N}$, then we divide the interval $[-1,1]$ into $n+1$ grids, each of width $h=\frac{2}{n+1}$,  such that the first $n$ grids are half open (on the right) while the $(n+1)^{\mathrm{st}}$ (last) grid is closed. 
\begin{itemize}
\item \emph{Grid Number:} Let $\iota : [-1,1] \to \{1,2,\ldots  , n+1\}$ denote the grid number of $x$,  $$\iota (x) = \left\lfloor \frac{x+1}{h}\right\rfloor +1$$ with $\iota(1) = n+1$ where $\left\lfloor \frac{x+1}{h}\right\rfloor$ denotes the largest integer not greater than $\frac{x+1}{h}$. 
\item \emph{Location within Grid:} Let $\lambda: [-1,1]  \to [0,1]$ denote the location within the grid of $x$ given by $$\lambda(x) = \frac{x+1}{h} - \big( \iota(x)-1\big)$$ with $\lambda(1)=1$. 
\item \emph{Grid Projection Operator:}
Let $L_1\left([-1,1];\mathbb{R}^{n+1}\right)$ denote the space of vector-valued integrable functions $v: [-1,1] \rightarrow \mathbb{R}^{n+1}$. The projection operator $\Pi_{n+1}:L_1[-1,1]\to L_1\left([0,1];\mathbb{R}^{n+1}\right)$ is defined by $$\left(\Pi_{n+1} f\right)_j(\lambda) =f((\lambda + j-1)h-1), f \in L_1 [-1,1],$$ where $\lambda \in [0,1]$ and $j \in \{1,2,\ldots , n+1\}$. 
\item \emph{Grid Embedding Operator:} The grid embedding operator $\Pi^{-1}_{n+1}$ embeds integrable functions defined on the grids onto $L_1[-1,1]$;  $$\left(\Pi^{-1}_{n+1} v\right)(x)=v_{\iota(x)}(\lambda (x)).$$ 
\end{itemize}
\end{definition}
Note that we can identify the space of continuous functions $C[-1,1]\subset L_1[-1,1]$ and that  $\Pi_{n+1}C[-1,1]= \left\{v \in C\left([0,1];\mathbb{R}^{n+1}\right) : v_{j+1}(0) = v_j (1) \mbox{ for } j = 1, \ldots ,n\right\}$ and that $\Pi^{-1}_{n+1}$ defined on the range is a closed and bounded operator on $X$. Moreover,
\begin{equation}
\label{projectembedidentityC0}
\Pi^{-1}_{n+1} \left(\Pi_{n+1} f\right) = f, \; f \in C[-1,1] \; \mathrm{and} \; \Pi_{n+1} \left(\Pi^{-1}_{n+1} v \right) = v, v \in \mathcal{R}(\Pi_{n+1}).
\end{equation}

\begin{definition}
\label{definitiontransitionoperatorG}

Let $G_{n \times n}$ denote a given $n\times n$ transition matrix on $l^n_\infty$. The \emph{transition operator} $G : C[-1,1] \rightarrow C [-1,1]$ is given by $$Gf(x) :=\left(\Pi^{-1}_{n+1} \left(G_{n+1} \Pi_{n+1}f\right)\right) (x) = \left[ G_{n+1}(\lambda (x)) (\Pi_{n+1}f)(\lambda (x))\right]_{\iota (x)},$$
where  the $(n+1)\times (n+1)$ \emph{interpolation matrix} $G_{n+1}$ is given by 
\begin{equation}
\label{interpolationmatrixG}
  G_{n+1}(\lambda)=\begin{pmatrix}
    g_{1,1}& D^l(\lambda) g_{1,2}& \cdots& D^l(\lambda)g_{1,n} & 0\\
    N^l(\lambda) g_{2,1}& & & &N^r(\lambda) g_{1,n}\\
    \vdots &  &  &  & \vdots \\
    N^l (\lambda) g_{i , 1} & \multicolumn{3}{c}{ (1-\lambda) g_{i-1,j-1}+\lambda g_{i,j}} & N^r(\lambda) g_{i-1, n}\\
    \vdots &  &  &  & \vdots \\
    N^l(\lambda) g_{n,1} &  & &  & N^r(\lambda) g_{n-1,n}\\
    0 & D^r(\lambda) g_{n,1} & \cdots & D^r(\lambda) g_{n,n-1}& g_{n,n}
  \end{pmatrix}.
\end{equation}
The parameter $\lambda \in [0,1]$, $g_{i, j}$ are the entries of $G_{n \times n}$, and $D^l, \; N^l, \; D^r, \; N^r$ are \textit{continuous interpolating functions} of the parameter $\lambda$ such that $G_{n+1}(\lambda)$ is also a rate matrix for each $\lambda \in [0,1]$. The interpolating functions are chosen in the following manner depending on the boundary conditions at hand. For $\mathrm{DR}$, $N^l=\mathbf{1}$ and $D^l$ is a continuous function of the parameter $\lambda$ that interpolates from $0$ to $1$; for other left boundary conditions we set $D^l=\mathbf{1}$ and $N^l$ interpolates from $0$ to $1$. Similarly, for $\mathrm{L D}$, $N^r=\mathbf{1}$ and $D^r$ interpolates from $1$ to $0$; for other boundary conditions on the right we set $D^r=\mathbf{1}$ and $N^r$ interpolates from $1$ to $0$.

\end{definition}

%
\begin{lemma} \label{remarkinterpolatingfunctionsandfellersemigroup}The transition operator $G$ is a bounded operator on $C_0(\Omega)$.\end{lemma}
\begin{proof}
The interpolating functions $D^l,\; D^r, \; N^l \; \mathrm{and}\; N^r$  ensure that the limit $\lim_{x\to x_b}Gf(x)=0$ for a Dirichlet boundary point $x_b\in [-1,1]\setminus\Omega$ as well as the continuity of $Gf$ at each grid point $x=ih-1,\; 1\le i\le n$. 
\end{proof}

In view of \eqref{projectembedidentityC0},
$$S(t)f:=e^{tG}f=\sum_{j=0}^\infty \frac{t^j}{j!}\left(\Pi^{-1}_{n+1} G_{n+1}\Pi_{n+1}\right)^j f= \Pi^{-1}_{n+1} e^{t G_{n+1}}\Pi_{n+1} f.$$ As $\Pi^{-1}_{n+1}$ and $\Pi_{n+1}$ are positive contractions, $(S(t))_{t \geq 0}$ is a Feller semigroup if the semigroup $(e^{t G_{n+1}})_{t \geq 0}$ is a positive contraction on $C\left([0,1];\mathbb{R}^{n+1}\right)$, which is the case if and only if $ G_{n+1}(\lambda)$ generates a positive contraction semigroup on $\ell^{n+1}_\infty$ for each $\lambda\in[0,1]$; that is, $ G_{n+1}(\lambda)$ is a rate matrix whose row sums are non-positive. 

\begin{remark}Note that the approximation operator usually used in numerical analysis that is obtained by linearly interpolating the vector $\tilde G \mathbf f $ is not the generator of a positive semigroup and therefore does not admit a straightforward stochastic interpretation.
\end{remark}

In the sequel, for a linear operator $A:\mathcal D(A) \subset X\to X$ and a subspace $Y\subset X$, we denote by $A\Big|_Y$, and call it the \textit{part of $A$ in $Y$}, the maximal operator on $Y$ induced by $A$; that is, 
$$\mathcal D (A\Big|_Y)=\{y\in Y\cap \mathcal{D}(A): \quad Ay\in Y\}$$ and $A\Big|_Y y=Ay$ for $y\in \mathcal D (A\Big|_Y).$  

In applications, one is usually interested in observing the time evolution of the forward semigroup that acts on the space of bounded (complex) Borel measures, $\mathcal{M}_{\mathcal{B}}(\Omega)$. It is well known that $L_1(\Omega)$ is isometrically isomorphic to the closed subspace of $\mathcal{M}_{\mathcal{B}}(\Omega)$ which consists of measures that possess a density. The forward semigroup denoted by $(T(t))_{t \geq 0}$ is the adjoint of $(S(t))_{t \geq 0}$ and the action of the generator $G^*$ of $(T(t))_{t \geq 0}$ can be easily computed for $g \in L_1(\Omega):=L_1[-1,1]$. To this end a simple calculation reveals $$ \int_{-1}^1 (Gf)(x) g(x)\,  \dd x = \int_{-1}^1 f(x) (G^*g)(x) \,  \dd x,$$ where $G^* g = \Pi^{-1}_{n+1} G^T_{n+1} \Pi_{n+1}g$. Thus, $G^*$ leaves $L_1[-1,1]$ invariant and we have the following proposition.

\begin{proposition}
\label{propositiontransitionoperatorG*}
The restriction of the adjoint transition operator to $L_1[-1,1]$ is a bounded linear operator $L_1[-1,1]\to L_1[-1,1]$ and hence it is the part $G^*\Big|_{L_1[-1,1]} $ of $G^*$ in $L_1[-1,1]$. It is given by
 $$G^*f (x) =\left(\Pi^{-1}_{n+1}G^T_{n+1} \Pi_{n+1}\right) f (x) = \left[ G^T_{n+1}(\lambda (x)) (\Pi_{n+1}f)(\lambda (x))\right]_{\iota (x)}$$ where $G^T_{n+1}$ is the transpose of $G_{n+1}$ given by \eqref{interpolationmatrixG}.
\end{proposition}

\begin{proposition}
\label{tildeA}
Let $G_{n \times n}$ be an $n\times n$ rate matrix with non-negative off-diagonal entries and non-positive row sums. For $\lambda \in [0,1]$, let the operator $G$ be as in Definition \ref{definitiontransitionoperatorG} and assume that the interpolating functions $D^l, D^r, N^l \; \mathrm{and} \; N^r$ are such that $G_{n+1}(\lambda)$ is a rate matrix for each $\lambda\in [0,1]$ and $G$ is a bounded operator on $C_0(\Omega)$. Then $G$ generates a Feller semigroup on $C_0(\Omega)$ and the restriction of $G^*$ to $L_1[-1,1]$ generates a strongly continuous positive contraction semigroup on $L_1[-1,1]$.
\end{proposition}
\begin{proof}
Let $\left(S(t)\right)_{t\geq 0}$ denote the semigroup generated by $G$ and $\left(T(t)\right)_{t\geq 0}$ denote the dual semigroup generated by $G^*$. Firstly, note that $\left(S(t)\right)_{t\geq 0}$ is a Feller semigroup on $C_0(\Omega)$ in view of Lemma \ref{remarkinterpolatingfunctionsandfellersemigroup}. Indeed, since $G$ is a bounded operator, $(S(t))_{t \geq 0 }$ is strongly continuous. The fact that $G_{n+1}(\lambda)$ is a transition rate matrix with non-positive row sums for each $\lambda \in [0,1]$ yields the fact $(S(t))_{t \geq 0 }$ is a contraction semigroup. Lastly, positivity follows in view of the linear version of Kamke's theorem, \cite[p. 124]{Arendt1986}; that is, $e^{tG_{n+1}}\geq 0 $ if and only if $g_{i,j}\geq 0 $ for $i \neq j$. The same argument yields the positivity of $(T(t))_{t \geq 0}$. Since $(S(t))_{t \geq 0 }$ is a contraction semigroup and for all $t\geq0$, $\left\|T(t)\right\|=\left\|S(t)\right\|$, we have that $\left(T(t)\right)_{t \geq 0}$ is a contraction semigroup on $\mathcal{M}_{\mathcal{B}}(\Omega)$. Next, note that $L_1[-1,1]$ is isometrically isomorphic to a closed subspace of  $\mathcal{M}_{\mathcal{B}}(\Omega)$. Hence, as $G^*$, and thus $\left(T(t)\right)_{t \geq 0}$ leaves $L_1[-1,1]$ invariant, the part of $G^*$ in $L_1[-1,1]$, which is the restriction of $G^*$ to $L_1[-1,1]$ by Proposition \ref{propositiontransitionoperatorG*}, is the generator of $\left(T(t)\Big|_{L_1 [-1,1]}\right)_{t \geq 0}$, see \cite[p. 43, 61]{Engel2000}.
\qquad \end{proof}

The method is now to let $n$ go to infinity and show convergence of the respective semigroups, which shows convergence of the respective well-understood processes. Note that the interpolated process $(\tilde X_t^n)_{t\ge 0}$ conditioned on $\tilde X_0^n=jh-1$ can be identified with the process $(X_t^n)_{t\ge 0}$ conditioned on $X_0^n=j$. In other words, the interpolation preserves the finite state process associated with the transition rate matrix on the grid points $jh-1$.

\section{Gr\"unwald-type approximations and the associated processes}
\label{sectionnumericalscheme}

In the following, let $1< \alpha \le 2$, $n \in \mathbb{N}$ and $h = \frac{2}{n+1}$. Further, let 
\begin{equation*}\mathcal{G}^\alpha_k = (-1)^k  {\binom{\alpha}{k}}, \;k \in \mathbb{Z},\end{equation*}
denote the \textit{Gr\"unwald coefficients}, which satisfy the following properties, see \cite{Oldham1974} for details:
 \begin{equation}\begin{split}
\label{grunwaldproperties} 
& \mathcal{G}_{n+1}^\alpha=\frac{n-\alpha}{n+1} \mathcal{G}_{n}^\alpha,\quad \mathcal{G}^\alpha_0 =1,\quad \mathcal{G}^\alpha_{-n} = 0\; \mathrm{for} \; n \in \mathbb{N}, \\
& \sum_{n=0}^\infty \mathcal{G}^\alpha_n =0, \quad \sum_{n=0}^k \mathcal{G}^\alpha_n =\mathcal{G}^{\alpha -1}_k, \quad \sum_{n=0}^{k} \mathcal{G}^q_{n} \mathcal{G}^Q_{k-n} = \mathcal{G}^{q+Q}_k  \\
& \mathrm{and} \; \mathcal{G}_n^\alpha=\frac{n^{-1-\alpha}}{\Gamma(-\alpha)} \big(1+O(n^{-1})\big), \quad \mathrm{as}\; n \to \infty. 
\end{split}
\end{equation}
In particular, for $1<\alpha\le2$, $\mathcal{G}_{1}^\alpha=-\alpha<0$ and $\mathcal{G}_{k\neq 1}^\alpha\ge 0$.

The shifted Gr\"unwald formula
\begin{equation}\label{Grunwald}D^\alpha_\pm f(x)=\lim_{h\to0}A_hf(x)=\lim_{h\to 0}\frac1{h^\alpha}\sum_{k=0}^\infty \mathcal{G}^\alpha_k f(x\pm(h-kh))\end{equation}
is well known and studied, see, for example, \cite{Tadjeran2006}. The operators $A_h$ are bounded Fourier multipliers and generate positive bounded contraction semigroups on $L_1(\R)$ and, by transference principle, on $C_0(\R)$ \cite{Baeumer2012a}. Extending the ideas of the previous section, dividing $\R$ up into grids of length $h$ and considering the space of two-sided sequences going to zero at infinity, $c_0(\mathbb Z)$, we define the projection  operator
$\Pi:C_0(\R)\to C([0,1);c_0(\mathbb Z))$ via 
$$\left(\Pi f\right)_j(\lambda)=f((j+\lambda)h)$$
where $j\in\mathbb Z$, $f\in C_0(\R)$, and $\lambda\in [0,1)$.  Then
$$A_hf=\Pi^{-1}G_h\Pi f$$ with
\begin{equation*}
G_h=
  \frac{1}{h^\alpha}\begin{pmatrix} 
     \ddots  & \ddots & \ddots & \ddots  &  \ddots& \ddots & \ddots  \\
     \ddots &  \mathcal{G}^\alpha_1 & \mathcal{G}^\alpha_2 & \ddots & \mathcal{G}^\alpha_{n-1}  & \mathcal{G}^\alpha_{n}  & \ddots   \\
     \ddots & \mathcal{G}^\alpha_0 & \mathcal{G}^\alpha_1 &    \ddots &\mathcal{G}^\alpha_{n-2} & \mathcal{G}^\alpha_{n-1}&  \ddots \\
     \ddots & 0&\mathcal{G}^\alpha_0&\ddots    & \ddots &\ddots & \ddots   \\
    \ddots  & \ddots&\ddots&\ddots&\mathcal{G}^\alpha_1&\mathcal{G}^\alpha_2& \ddots  \\
    \ddots & 0 &\ddots&0&\mathcal{G}^\alpha_0&\mathcal{G}^\alpha_1&  \ddots \\
    \ddots  & \ddots  & \ddots & \ddots  &\ddots & \ddots& \ddots 
  \end{pmatrix}.
\end{equation*}
Note that the entries of $G_h$ are negative on the main diagonal and positive everywhere else and that the row sums and column sums add to zero; i.e. $G_h$ is a transition rate matrix on $h\mathbb Z$, perfectly describing the stochastic processes $X^h_t$ with generators $A_h$. In particular, if $X^h_0=(j+\lambda)h$, then $X_t^h\in \lambda h+h\mathbb Z$ for all $t\ge 0$.  We call this process the \emph{Gr\"unwald process}.

Restricting this process to a finite domain is at the heart of this article. Philosophically, boundary conditions should only influence the process at the boundary; i.e., if the process moves across a boundary, it can be restarted somewhere (or killed). We will therefore restrict ourselves to finite state processes with transition rate matrix being the central square of $G_h$, where we can modify the first and last row and column to suit a particular boundary condition. 

For the six cases of Table \ref{explicitProcesses}; i.e., killing, fast-forwarding, or reflecting at the boundary, we modify the countable state transition matrix such that the resulting finite state process is obtained by killing, fast-forwarding, or reflecting at the respective boundary. Recall that the entry $g_{i,j}$ represents the rate at which particles move from state $i$ to state $j$.

\begin{itemize}
\item[$(\mathrm{DD})$] For the process to be killed at either boundary, nothing needs to be changed in the central square of $G_h$ to get $G_{n\times n}^{\mathrm{DD}}$:
\begin{equation*}
G_h^{\mathrm{DD}}=
  \frac{1}{h^\alpha}\begin{pmatrix} 
       \mathcal{G}^\alpha_1 & \mathcal{G}^\alpha_2 & \cdots & \mathcal{G}^\alpha_{n-1}  & \mathcal{G}^\alpha_{n}    \\
    \mathcal{G}^\alpha_0 & \mathcal{G}^\alpha_1 &    \ddots &\mathcal{G}^\alpha_{n-2} & \mathcal{G}^\alpha_{n-1} \\
     0&\mathcal{G}^\alpha_0&\ddots    & \ddots &\vdots    \\
    \vdots&\ddots&\ddots&\mathcal{G}^\alpha_1&\mathcal{G}^\alpha_2  \\
    0 &\cdots&0&\mathcal{G}^\alpha_0&\mathcal{G}^\alpha_1\\
  \end{pmatrix}.
\end{equation*}
This is because killing doesn't change the rate at which particles leave a state or jump to a certain state. It only changes the states available.
\item[$(\mathrm{N^*N})$] Restarting at the left and right boundary if it is crossed respectively (called reflecting in probability theory) necessitates that we change the first element and the last row. The only particles that jump across the left boundary are the ones that $X_t^h$ would move from state one to state zero. These particles are now restarted in state one, reducing the rate of particles leaving state one by the rate at which particles would have moved to state zero; i.e., $$g^{\mathrm{N^*R}}_{1,1}=  \mathcal{G}^\alpha_1+\mathcal{G}^\alpha_0=\mathcal{G}^{\alpha-1}_1.$$ 

On the right boundary we need to collect all the particles that would have jumped from a state to a state beyond the right boundary and add them to the rate at which particles arrive from that state; i.e., 
$$g^{\mathrm{LN}}_{i,n}=\mathcal{G}^\alpha_{n+1-i}+\sum_{k=n+2-i}^\infty \mathcal{G}^\alpha_k=-\mathcal{G}^{\alpha-1}_{n-i}$$ resulting in
\begin{equation*}
G_{n\times n}^{\mathrm{N^*N}}=
  \frac{1}{h^\alpha}\begin{pmatrix} 
       \mathcal{G}^{\alpha-1}_1 & \mathcal{G}^\alpha_2 & \cdots & \mathcal{G}^\alpha_{n-1}  & -\mathcal{G}^{\alpha-1}_{n-1}    \\
    \mathcal{G}^\alpha_0 & \mathcal{G}^\alpha_1 &    \ddots &\mathcal{G}^\alpha_{n-2} &  -\mathcal{G}^{\alpha-1}_{n-2} \\
     0&\mathcal{G}^\alpha_0&\ddots    & \vdots &\vdots    \\
    \vdots&\ddots&\ddots&\mathcal{G}^\alpha_1& -\mathcal{G}^{\alpha-1}_1  \\
    0 &\cdots&0&\mathcal{G}^\alpha_0& -\mathcal{G}^{\alpha-1}_0\\
  \end{pmatrix}.
\end{equation*}
In particular, note that the row sums are all zero, showing that the process is conservative.

\item[$(\mathrm{NR})$] Lastly we model the case where particles that move across the left boundary are immediately restarted at the place of first re-entry (fast-forward).  Theorem \ref{lambdaRlambda*} below shows that the probability of restarting in state $j$ is given by $ -\mathcal{G}^{\alpha-1}_{j}$ and as the rate of moving from state one to state zero is $\frac1{h^\alpha}\mathcal{G}^{\alpha}_0=1/h^\alpha$, 
$$g^{\mathrm{NR}}_{1,j}= \mathcal{G}^{\alpha}_{j}-\mathcal{G}^{\alpha-1}_{j}=-\mathcal{G}^{\alpha-1}_{j-1},$$ resulting in, for example, 
\begin{equation*}
G_{n\times n}^{\mathrm{NN}}=
  \frac{1}{h^\alpha}\begin{pmatrix} 
       -\mathcal{G}^{\alpha-1}_0 & -\mathcal{G}^{\alpha-1}_1 & \cdots & -\mathcal{G}^{\alpha-1}_{n-2}  & \mathcal{G}^{\alpha-2}_{n-2}    \\
    \mathcal{G}^\alpha_0 & \mathcal{G}^\alpha_1 &    \cdots &\mathcal{G}^\alpha_{n-2} & -\mathcal{G}^{\alpha-1}_{n-2} \\
     0&\mathcal{G}^\alpha_0&\ddots    & \vdots &\vdots    \\
    \vdots&\ddots&\ddots&\mathcal{G}^\alpha_1&-\mathcal{G}^{\alpha-1}_1 \\
    0 &\cdots&0&\mathcal{G}^\alpha_0& -\mathcal{G}^{\alpha-1}_0\\
  \end{pmatrix}.
\end{equation*}

\end{itemize}
Other combinations of boundary conditions are defined analogously and summarised in Table \ref{mainTable} using the generic matrix
\begin{equation}\label{Glr}
G_{n\times n}^{\mathrm{LR}}=
  \frac{1}{h^\alpha}\begin{pmatrix} 
       b_1^l & b_2^l & \cdots & b_{n-1}^l & b_n    \\
    \mathcal{G}^\alpha_0 & \mathcal{G}^\alpha_1 &    \cdots &\mathcal{G}^\alpha_{n-2} & b_{n-1}^r \\
     0&\mathcal{G}^\alpha_0&\ddots    & \vdots &\vdots    \\
    \vdots&\ddots&\ddots&\mathcal{G}^\alpha_1&b_2^r \\
    0 &\cdots&0&\mathcal{G}^\alpha_0& b_1^r\\
  \end{pmatrix}.
\end{equation}
\begin{table}
\centering
\vline
\begin{tabular}{c|c|c|c|}
  \hline
  Case  & $b^{l,r}$ &$D^{l,r}$& $N^{l,r}$ \\ 
  \hline
  $(\mathrm{DR})$ & $b_i^l=\mathcal G_i^\alpha$ & $D^l(\lambda)=\frac{ \lambda \alpha }{1-\lambda +\lambda \alpha}$ & $N^l=\mathbf{1}$\\
  $(\mathrm{NR})$ & $b_i^l=-\mathcal G_{i-1}^{\alpha-1}$ & $D^l=\mathbf{1}$ & $N^l(\lambda)=\lambda$\\
  $(\mathrm{N^*R})$ & $b_0^l=0,b_1^l=\mathcal G_{1}^{\alpha-1},b_i^l=\mathcal G_i^\alpha,i\ge2$ & $D^l=\mathbf{1}$ & $N^l(\lambda)=\lambda$\\
  $(\mathrm{LD})$ & $b_i^r=\mathcal G_i^\alpha$, $b_n=b_n^l$ & $D^r(\lambda)=\frac{ (1-\lambda) \alpha }{\lambda+  (1-\lambda )\alpha}$&$N^r=\mathbf{1}$\\
  $(\mathrm{LN})$ & $b_i^r=-\mathcal G_{i-1}^{\alpha-1}, b_n=-\sum_{i=0}^{n-1}b_i^l$& $D^r=\mathbf{1}$&$N^r(\lambda)=1-\lambda$\\
\hline
\end{tabular}
\caption{\label{mainTable}
Table of boundary weights and interpolating functions used to build the interpolation matrix \eqref{interpolationmatrixGBC} for the transition operators $G_{\pm h}^{\mathrm{LR}}$ of \eqref{G-hBC} and \eqref{G+hBC}.}
\end{table}

Armed with the transition rate matrices for the finite state processes we build the interpolation matrices introduced in \eqref{interpolationmatrixG}. That is, we need to define the functions $D^{l,r}$ and $N^{l,r}$ for the respective six cases and we choose them according to  Table \ref{mainTable} to obtain the interpolation matrices
\begin{equation}
\label{interpolationmatrixGBC}
  G^{\BC}_{n+1}(\lambda)=\frac1{h^\alpha}\begin{pmatrix}
    b_1^l& D^l(\lambda) b_2^l&\cdots& \cdots&D^l(\lambda)b_n & 0\\
    N^l(\lambda) \mathcal G_0^\alpha & \lambda \mathcal G_1^\alpha+\lambda'b_1^l &\cdots&\cdots&\lambda b^r_{n-1}+\lambda'b_{n-1}^l & N^r(\lambda) b_n\\
    0 & \mathcal G_0^\alpha & \cdots&\mathcal G_{n-3}^\alpha&\vdots & \vdots \\
    \vdots & \ddots &\ddots& \vdots &\vdots&\vdots\\
   \vdots &  &\ddots&\mathcal G_0^\alpha & \lambda b_1^r+\lambda'\mathcal G_1^\alpha & N^r(\lambda) b_2^r\\
    0 &\cdots &\cdots& 0& D^r(\lambda) \mathcal G_0^\alpha & b_1^r
  \end{pmatrix}
\end{equation}
with $\lambda'=1-\lambda$, leading to the transition operators
\begin{equation}\label{G-hBC}G_{-h}^{\BC}f(x)  = \left[ G_{n+1}^{\BC}(\lambda (x)) (\Pi_{n+1}f)(\lambda (x))\right]_{\iota (x)}
\end{equation}
on $C_0(\Omega)$ and
\begin{equation}\label{G+hBC}G_{+h}^{\BC}f(x)  = \left[ \left(G_{n+1}^{\BC}(\lambda (x))\right)^T (\Pi_{n+1}f)(\lambda (x))\right]_{\iota (x)}
\end{equation}
on $L_1[-1,1]$.

\subsection{Details of the case $(\mathrm{NR})$}

To determine the probability of a certain positive state being the first being visited by $X^h_t$ started at $X^h_0=0$, we consider the long term distribution of the stopped process starting at $0$; i.e., let $$\xi_+(t)=\min\{t:X^h_t>0\}$$ and consider
$$X^h_{\mathrm {stop}}(t)=\begin{cases} X^h_t&t<\xi_+(t)\\X^h_{\xi_+(t)}&t\ge\xi_+(t)\end{cases}.$$
Its generator (transition rate matrix) is given by turning the rates for leaving state $i$ for positive $i$ off: $$G_{\mathrm{stop}}=(g_{i,j})_{i,j\in\mathbb Z}\mbox{ with }g_{i,j}=\frac1{h^\alpha}\begin{cases}\mathcal G^\alpha_{j-i+1}&i\le 0\\0&i> 0.\end{cases}$$
By the Abelian Theorem for Laplace transforms \cite[Thm 4.1.2] {Arendt2001}the steady state of the stopped process starting at zero is given by 
$$\lim_{t\to \infty}e^{tG_{\mathrm{stop}}^*}\vec{e}_{0}=\lim_{\lambda\to 0+} \lambda(\lambda-G_{\mathrm{stop}}^*) ^{-1} \vec{e}_{0},$$ where $\vec{e}_{0}$ is the vector with $x_i=0$ for all $i\neq 0$ and $x_{0}=1$.

\begin{theorem}
The resolvent of $G^*_{\mathrm{stop}}$ evaluated at $\vec e_{0}$ is given by
\begin{equation}\label{resy*}
\left((\lambda I-G^*_{\mathrm{stop}})^{-1}\vec e_{0}\right)_n=\begin{cases}e^{(n-1)\psi^{-1}(\lambda)}&n\le 0\\
\frac1\lambda \sum_{k=n}^\infty \mathcal G_{k+1}^\alpha e^{-(k-n+1)\psi^{-1}(\lambda)} &n> 0\end{cases},
\end{equation}
where $$\psi(\lambda)=e^\lambda(1-e^{-\lambda})^\alpha=e^\lambda \sum_{k=0}^\infty (-1)^k\begin{pmatrix}\alpha\\k\end{pmatrix}e^{-\lambda k}= \sum_{k=0}^\infty \mathcal G^\alpha_ke^{(1-k)\lambda }.$$
\end{theorem}

\begin{proof}
First we show that $\psi$ is invertible by showing that it is increasing for $\lambda>0$. This follows from
$$\psi'(\lambda)=(1-e^{-\lambda})^{\alpha-1}\left(e^\lambda-1+\alpha\right)>0.$$

Next we apply  $\lambda I-G^*_{\mathrm{stop}}$ to our resolvent candidate and show that the result is indeed $\vec e_{0}$. Let $y$ be given by \eqref{resy*}. For $n\le -1$,
\begin{equation}\begin{split}\left((\lambda I-G_{\mathrm{stop}}^*)y\right)_n=&\lambda e^{(n-1)\psi^{-1}(\lambda)}-\sum_{k=0}^\infty \mathcal G_k^\alpha e^{(n-k)\psi^{-1}(\lambda)}\\
=&\lambda e^{(n-1)\psi^{-1}(\lambda)}-e^{(n-1)\psi^{-1}(\lambda)}\sum_{k=0}^\infty \mathcal G_k^\alpha e^{(1-k)\psi^{-1}(\lambda)}\\
=&\lambda e^{(n-1)\psi^{-1}(\lambda)}-e^{(n-1)\psi^{-1}(\lambda)}\psi(\psi^{-1}(\lambda))\\
=&0.
\end{split}\end{equation}
For $n=0$, 
\begin{equation}\begin{split}\left((\lambda I-G^*_{\mathrm{stop}})y\right)_{0}=&\lambda e^{-\psi^{-1}(\lambda)}-\sum_{k=1}^\infty \mathcal G_k^\alpha e^{-k\psi^{-1}(\lambda)}\\
=&\lambda e^{-\psi^{-1}(\lambda)}-e^{-\psi^{-1}(\lambda)}\sum_{k=0}^\infty \mathcal G_k^\alpha e^{(1-k)\psi^{-1}(\lambda)}+ \mathcal G_0^\alpha \\
=&\lambda e^{-\psi^{-1}(\lambda)}-e^{-\psi^{-1}(\lambda)}\psi(\psi^{-1}(\lambda))+1\\
=&1.
\end{split}\end{equation}
For $n\ge 1$, $\left((\lambda I-G^*_{\mathrm{stop}})y\right)_{n}=0$ by the  definition of $G^*_{\mathrm{stop}}$. Hence $(\lambda I-G^*_{\mathrm{stop}})y=\vec e_{0}$ and therefore $(\lambda I-G^*_{\mathrm{stop}})^{-1}\vec e_{0}=y$.
\end{proof}

\begin{corollary}\label{lambdaRlambda*} The steady state probability distribution of the stopped process started at zero is given by
$$\lim_{\lambda\to 0} \lambda(\lambda-G^*_{\mathrm{stop}}) ^{-1} \vec{e}_{0}=\vec z,$$
where $z_i=0$ for $i\le 0$ and $z_i=-\mathcal G^{\alpha-1}_{i}$ for $i> 0$.
\end{corollary}
\begin{proof}As $\lambda\to 0$, $\psi^{-1}(\lambda)\to 0$ and hence 
$$\left(\lambda R(\lambda,G^*_{\mathrm{stop}})\vec e_{0}\right)_n\to 0$$ for all $n\le 0$. For $n> 0$,
$$\left(\lambda R(\lambda,G^*_{\mathrm{stop}})\vec e_{0}\right)_n\to \sum_{k=n}^\infty \mathcal G_{k+1}^\alpha=-\mathcal G_{n}^{\alpha-1}.$$
\quad
\end{proof}

\section{Convergence of the semigroups and processes}
\label{subsectiontrotterkatotheorem}
In this section we use the Trotter-Kato Theorem to show convergence of the semigroups generated by the transition operators $G_{-h}^{\mathrm{LR}}$ based on the interpolation matrices defined in Table \ref{mainTable}. This will imply convergence of the associated stochastic processes $X^h_t$. As the processes $X_t^h$ started at a grid point correspond to the modified Gr\"unwald process according to Table \ref{explicitProcesses}, and the modifications are continuous with respect to the Skorokhod metric, the limit processes are the processes of Table \ref{explicitProcesses}.

First we ensure that the transition operators in \eqref{G-hBC} and \eqref{G+hBC} are indeed suitable.
\begin{lemma}
\label{Ghdissipative}
The transition operators $G^{\mathrm{LR}}_{\mp h}$ defined in \eqref{G-hBC} and \eqref{G+hBC} via Table \ref{mainTable} are generators  of positive strongly continuous contraction semigroups on $C_0(\Omega)$ (i.e. Feller semigroups) or on $L_1[-1,1]$, respectively. 
\end{lemma}
\begin{proof}
 In view of Proposition \ref{tildeA} all we have to show is that the interpolation matrices $G_{n+1}^{\BC}(\lambda)$ of \eqref{interpolationmatrixGBC} are rate matrices; i.e., that for each $\lambda$ the row sums are not positive. This is straight forward for each row, except maybe the second row which adds to
 $$S_2^{\mathrm{LR}}=N^l(\lambda)\mathcal G_0^\alpha+\lambda\sum_{k=1}^{n-2}\mathcal G_k^\alpha+\lambda b_{n-1}^r+(1-\lambda)\sum_{k=1}^{n-1}b_k^l+N^r(\lambda)b_n.$$
 In case of $(\mathrm{LD})$ substituting $b_n$ and $b^r_i$ from Table \ref{mainTable}, this simplifies to 
 $$S_2^{\mathrm{LD}}=N^l(\lambda)\mathcal G_0^\alpha+\lambda\sum_{k=1}^{n-1}\mathcal G_k^\alpha+(1-\lambda)\sum_{k=1}^{n-1}b_k^l+b_n^l$$ which further reduces for $(\mathrm{ND})$ to
 $$S_2^{\mathrm{ND}}=\lambda\sum_{k=0}^{n-1}\mathcal G_k^\alpha-(1-\lambda)\sum_{k=1}^{n-1}\mathcal G_{k-1}^{\alpha-1}-\mathcal G_{n-1}^{\alpha-1}=-(1-\lambda)\mathcal G_{n-1}^{\alpha-2}<0$$
  and it is easily verified that for the other two left boundary conditions  $S_2< 0$ as well.
 
 In case of $(\mathrm{LN})$ the second row sum simplifies to
 $$S_2^{\mathrm{LN}}=N^l(\lambda)\mathcal G_0^\alpha+\lambda\sum_{k=1}^{n-2}\mathcal G_k^\alpha-\lambda \mathcal G^{\alpha-1}_{n-2}-(1-\lambda)b_0^l=0.$$

Therefore $G_{n+1}(\lambda)$ is a rate matrix for all $0\le\lambda\le1$ and by  Proposition \ref{tildeA} the proof is complete.
\end{proof}

Next we show that the operators converge; i.e. that $G_{-h}^{\mathrm{LR}}\to (A^-,\mathrm{LR})$ and  $G_{+h}^{\mathrm{LR}}\to (A^+,\mathrm{LR})$ for the twelve cases in Table \ref{explicitdomains}. In particular, we show that for each $f\in\D(A,\mathrm{LR})$ there exists $f_h$ such that $f_h\to f$ and $G_{\pm h}^{\mathrm{LR}}f_h\to A^\pm f $. 
\begin{proposition}
\label{mainproposition}
Let $(A,\BC)$ be one of the operators of Table \ref{explicitdomains} with domain $\mathcal D(A,\mathrm{LR})$. For each $f\in \mathcal D(A,\mathrm{LR})$ there exists $f_h\in X$ such that $f_h\to f$ and $G_{-h}^{\BC}f_h\to Af$ if $X=C_0(\Omega)$ or $G_{+h}^{\BC}f_h\to Af$ if $X=L_1[-1,1]$.
\end{proposition}
\begin{proof}
Note that is is enough to show this property for each element in the core $\mathcal C(A , \BC )$ given by \eqref{coreA}. 
As $G_h^{\BC}p_\beta$ for $\beta\in\{\alpha,\alpha-1,\alpha-2\}$ does not converge in general we give an explicit sequence for each of the twelve cases. See Section \ref{sectionmainproposition} for details.
\end{proof}

\begin{theorem}[\bf Trotter-Kato type approximation theorem]
\label{trotterkatotypetheorem} The operators of Table \ref{explicitdomains} generate
 positive, strongly continuous, contraction semigroups on $X$. Moreover, the semigroups generated by $G_{-h}^{\BC}$ converge strongly (and uniformly for $t\in[0,t_0]$) to the semigroup generated by $(A^-, \BC)$ on $C_0(\Omega)$ and the semigroups generated by $G_{+h}^{ \BC}$ converge strongly (and uniformly for $t\in[0,t_0]$) to the semigroup generated by $(A^+, \BC)$ on $L_1[-1,1]$.
\end{theorem}
\begin{proof}
By Proposition \ref{mainproposition}, for each $f\in \mathcal D(A,\mathrm{LR})$ there exist sequences $\left\{f_h\right\} \subset C_0(\Omega)$ and $\left\{f_h\right\} \subset L_1[0,1]$ such that $f_h\to f$ and $G^h f_h \to Af$ in the respective norms. In view of Lemma \ref{Ghdissipative}, $G^h$ are dissipative; that is, $\|(\lambda-G^h)f_h\|\ge \lambda \|f_h\|$ for all $f_h\in X$ and all $\lambda>0$. Thus, as $h \to 0$, in view of Proposition \ref{mainproposition} we have $\|(\lambda-A)f\|\ge \lambda \|f\|$ for all $f\in \mathcal D(A,\mathrm{LR})$ in the respective $X$-norms and hence $(A, \BC)$ are dissipative. Furthermore, in view of Theorems \ref{theoremAdenseA}, \ref{theoremClosedA} and \ref{denseRangeIminusA}, $(A, \BC)$ are densely defined closed operators with dense $\mathrm{rg}(\lambda -A)$. Hence, the operators $(A, \BC)$ generate strongly continuous contraction semigroups as a consequence of the Lumer-Phillips Theorem. The second statement and the positivity of the semigroups generated by  $(A, \BC)$ follow using the Trotter-Kato Theorem in view of Proposition \ref{mainproposition} and Lemma \ref{Ghdissipative}.
\qquad \end{proof} 
\begin{corollary}
\label{adjointfractionalderivatives}
Let $(A^+, \BC)$ on $L_1[-1,1]$ and $(A^-, \BC)$ on $C_0(\Omega)$ be as in Table \ref{explicitdomains}. Then, $(A^-, \BC) = (A^+, \BC)^*\Big|_{C_0(\Omega)}$ and $(A^+, \BC) = (A^-, \BC)^*\Big|_{L_1[-1,1]}$.
\end{corollary}
Recall that for a linear operator $A:\mathcal D(A) \subset X\to X$ and a subspace $Y\subset X$, we denote by $A\Big|_Y$ the part of $A$ in $Y$.
\begin{proof}
Given the explicit representation of the respective domains we first show that the operators satisfy $(A^\pm,\BC)\subset (A^\mp,\BC))^*$; i.e. we show that for $\phi_1\in \D(A^+,\BC)$ and $\phi_0\in\D(A^-,\BC)$,
$$\Delta:=\int_{-1}^1 \phi_0(x)A^+\phi_1(x)\,dx-\int_{-1}^1\phi_1(x)A^-\phi_0(x)\,dx=0.$$ 

Let $$\phi_0=I^\alpha_- g_0+a_0p^-_\alpha+b_0p^-_{\alpha-1}+c_0p_0\in \D(A^-,\BC)$$ and
$$\phi_1=I^\alpha_+ g_1+a_1p^+_\alpha+b_1p^+_{\alpha-1}+c_1p^+_\eta\in \D(A^+,\BC)$$
for one of the six cases of Table \ref{explicitdomains}.
As $\int_{-1}^1 f I^\alpha_+ g= \int_{-1}^1gI^\alpha_-f$ and $\int_{-1}^1 p_{\beta}^\pm f=I^{\beta+1}_\mp f(\mp 1)$, 
\begin{equation}\begin{split}\Delta=&\int_{-1}^1\phi_0(x)\left(g_1(x)+a_1\right)-\phi_1(x)\left(g_0(x)+a_0\right)\,dx\\
=&\int_{-1}^1 \left(I^{\alpha}_-g_0+a_0p^-_\alpha+b_0p^-_{\alpha-1}+c_0p_0\right)a_1\,dx\\&-\int_{-1}^1\left(I_+^\alpha g_1+a_1p^+_\alpha+b_1p^+_{\alpha-1}+c_1p^+_\eta\right)a_0\,dx\\
&+\int_{-1}^1 \left(a_0p^-_\alpha+b_0p^-_{\alpha-1}+c_0p_0\right)g_1(x)-\left(a_1p^+_\alpha+b_1p^+_{\alpha-1}+c_1p^+_\eta\right)g_0(x)\,dx\\
=&a_1I^{\alpha+1}_-g_0(-1)-a_0I^{\alpha+1}_+g_1(1)+b_0a_1p_\alpha^-(-1)-b_1a_0p_\alpha^+(1)+c_0a_1p_1^-(-1)\\&-c_1a_0p_{\eta+1}(1)+a_0I^{\alpha+1}_+g_1(1)+b_0I^{\alpha}_+g_1(1)+c_0Ig_1(1)\\&-a_1I^{\alpha+1}_-g_0(-1)-b_1I^\alpha_-g_0(-1)-c_1I^{\eta+1}_-g_0(-1)\\
=&b_0\left(I^\alpha_+g_1(1)+a_1p_\alpha^-(-1)\right)+c_0\left(Ig_1(1)+p_1^-(-1)a_1\right)\\
&-b_1\left(I^\alpha_-g_0(-1)+a_0p_\alpha^+(1)\right)-c_1\left(I_-^{\eta+1}g_0(-1)+a_0p_{\eta+1}(1)\right).
\end{split}
\end{equation}
Then for each of the six cases of Table \ref{explicitdomains}  one can verify that $\Delta=0$. Thus, $(A^-, \BC) \subset (A^+, \BC)^*\Big|_{C_0(\Omega)} $ and $ (A^+, \BC) \subset (A^-, \BC)^*\Big|_{L_1[-1,1]}$. 

As a consequence, 
$$I - (A^-, \BC) \subset \big( I- (A^+, \BC)^*\big)\Big|_{C_0(\Omega)}$$ and  $$ I- (A^+, \BC) \subset \big(I- (A^-, \BC)^*\big)\Big|_{L_1[-1,1]},$$ 
where the identity operator on the respective spaces is denoted by $I$. Moreover, in view of Theorem \ref{trotterkatotypetheorem}, the operators $(A^+, \BC)$ and $(A^-, \BC)$ generate strongly continuous semigroups on $L_1[-1,1]$ and $C_0(\Omega)$, respectively. Therefore, in particular, $1 \in \rho((A^+, \BC))$ and $1 \in \rho((A^-, \BC))$. Thus, $I - (A^-, \BC)$ is surjective and $I - (A^+, \BC)^*$ is injective where the latter implies that $\big(I - (A^+, \BC)^*\big)\Big|_{L_1[-1,1]}$ is also injective. This yields $$I - (A^-, \BC) = \big( I - (A^+, \BC)^*\big)\Big|_{C_0(\Omega)}$$ since if operators $T, S$ are such that $T \subset S$, $T$ is surjective and $S$ is injective, then $T=S$. Hence, $(A^-, \BC) = (A^+, \BC)^*\Big|_{C_0(\Omega)}$. A similar argument holds for the pair, $(A^+, \BC)$ and $(A^-, \BC)^*\Big|_{L_1[-1,1]}$ . 
\qquad \end{proof}

\begin{corollary}
\label{processconvergence}
The stochastic processes $(X^h_t)_{t \geq 0}$ with generators $G_{-h}^{\BC}$ converge in the Skorokhod topology to the limit processes  whose (backwards) generators are the fractional derivative operators $(A^-, \BC)$ of Table \ref{explicitProcesses}. Furthermore, these limit processes are the processes $(Z_t)_{t \geq 0}$ of Table \ref{explicitProcesses}. 
\end{corollary}
\begin{proof}
 Theorem \ref{trotterkatotypetheorem}, Corollary \ref{adjointfractionalderivatives} and \cite[p. 331, Theorem 17.25]{Kallenberg1997} show that the stochastic processes $(X^h_t)_{t \geq 0}$ with generators $G_{-h}^{\BC}$ converge in the Skorokhod topology to the processes associated with the fractional derivative operators $(A^-, \BC)$ of Table \ref{explicitdomains}. It remains to show that these processes  are actually the processes $(Z_t)_{t \geq 0}$ obtained by modifying the process via killing, restarting, or fast-forwarding at the respective boundaries. Note that these modifications (via supremum, infimum, integrals, etc) are continuous mappings in the Skorokhod metric. We know that on $\R$ the discrete Gr\"unwald process converges in the Skorokhod topology to the stable process and hence these modifications of the discrete Gr\"unwald process will converge to the modification of the stable process.   For each starting point $x\in \Omega\cap\mathbb Q$ there exists $h^x$ such that $x$ is a grid point for all grids $h^x/n$.  We showed in Section \ref{sectionnumericalscheme} that  $(X^{h^x/n}_t)_{t \geq 0}$ starting at $X^{h^x/n}_0=x$ is the same process as the discrete modified Gr\"unwald process starting at $x$.  As the infimum, supremum, integrals and killing are continuous functions in the Skorokhod metric, the limit process is the limit of the modified Gr\"unwald process  started at $x\in \Omega\cap \mathbb Q$. Since these processes are Feller, this now also holds for all $x\in \Omega$ and therefore the limit processes are indeed the modified stable processes.
\qquad \end{proof}

\section{Proof of Proposition \ref{mainproposition}}
\label{sectionmainproposition}

We begin with the necessary preparations for the proof. For polynomials $P= \sum_{m=0}^N k_m p_m \in X$,  let
\begin{equation}
\mathcal{P}= \left\{ \begin{array}{lc}
I_+^\alpha P, \; &\mathrm{if} \; X = L_1[-1,1],\\
I_-^\alpha ( P - P(1) p_0), \;& \mathrm{if} \; X = C_0(\Omega).
\end{array}\right.
\end{equation}
Without loss of generality we consider the core $$\mathcal {C}(A, \BC)=\{f:f=\mathcal P+ap_\alpha+bp_{\alpha-1}+cp_\eta\},$$ where $a,b$, and $c$ are given in Tables \ref{fhexplicit1} and \ref{fhexplicitC}.
 As $G^{\BC}_h p_\beta$ for relevant $\beta\in\{\alpha,\alpha-1,\alpha-2,0\}$ does not converge to zero in the respective norms, we need to approximate these functions with functions tailored to $G_h^{\BC}$.

\begin{definition}
\label{varthetafunctions}
The approximate power functions are given by 
\begin{equation}
\label{varthetageneralform}
\vartheta_{ h}^{\beta}(x) = h^{\beta} \Big( \big(1- \theta (\lambda) \big) \mathcal{G}_{\iota( x)-2-\tau}^{-\beta-1}+ \theta (\lambda) \mathcal{G}_{\iota( x)-1-\tau}^{-\beta-1}\Big), \;\textnormal {if} \; \iota( x) \neq 1,
\end{equation}
where $\tau$, $\theta, $ and $\vartheta^\beta_h (x)$ for $\iota(x) =1$ are given in Table \ref{varthetadetails} for each $\beta$. We then define the space specific functions (`+' if $X=L_1[-1,1]$ and `-' if $X=C_0(\Omega)$) via
$$\vartheta^{\beta}_{\pm h}(x)=\vartheta^\beta_h(\pm x).$$
\end{definition}

\begin{table}
\centering
\vline
\begin{tabular}{c|c|c|}
  \hline
  $\beta$  & $L_1[-1,1]$ & $C_0(\Omega)$ \\ 
\hline
$\alpha$ & $\begin{matrix} \theta (\lambda)= 1,\; \tau =1\hfill\\
                                               \vartheta^\alpha_h (x) = - h^\alpha\lambda' \mathcal{G}^{-\alpha-1}_0 ,\; \mathrm{if} \;\iota (x) =1 \end{matrix}$ & $\begin{matrix} \theta (\lambda)= \lambda,\;\tau =1 \hfill\\
                                               \vartheta^\alpha_h (x) = - h^\alpha \lambda' \mathcal{G}^{-\alpha-1}_0 ,\; \mathrm{if} \; \iota (x) =1 \end{matrix}$\\
\hline
$\alpha-1$ & $\begin{matrix} \theta (\lambda)= \lambda,\;\tau =0 \hfill\\
                                               \vartheta^{\alpha-1}_h (x) = \frac{h^{\alpha-1}}{\alpha} (\lambda' \mathcal{G}^{-\alpha}_0 +\lambda \mathcal{G}^{-\alpha}_1) ,\; \mathrm{if} \;\iota (x) =1 \end{matrix}$ & $\begin{matrix} \theta (\lambda)= \lambda,\;\tau =0 \hfill\\
                                               \vartheta^{\alpha-1}_h (x) = h^{\alpha -1} \lambda \mathcal{G}^{-\alpha}_0 ,\; \mathrm{if} \; \iota (x) =1 \end{matrix}$\\
\hline
$0$ & $\begin{matrix} \theta (\lambda)= 1,\;\tau =0 \hfill\\
                                               \vartheta^0_h (x) = \lambda \mathcal{G}^{-1}_0 ,\; \mathrm{if} \;\iota (x) =1 \end{matrix}$ & $\begin{matrix} \theta (\lambda)= 1,\;\tau =0 \hfill\\
                                               \vartheta^0_h (x) = \mathcal{G}^{-1}_0 ,\; \mathrm{if} \; \iota (x) =1 \end{matrix}$\\
\hline
$\alpha -2$ & $\begin{matrix} \theta (\lambda)= \frac{\lambda}{(\alpha -1)\lambda' + \lambda}, \;\tau =0\hfill\\
                                               \vartheta^{\alpha-2}_h (x) = h^{\alpha -2} \theta (\lambda ) \mathcal{G}^{-\alpha +1}_0 ,\; \mathrm{if} \;\iota (x) =1 \end{matrix}$ & \\
\hline
\end{tabular}
\caption{\label{varthetadetails}
Approximate power functions on $X$.}
\end{table}
For each $f\in \mathcal {C}(A, \BC)$ 
we construct functions $f_h$ that converge to $f$ as well as $G_{\pm h}^{\BC}f_h\to A^\pm f$, where $f_h$ are of the form
\begin{equation}
\label{functionfh}
f_h= \mathcal {P}+a_h\vartheta_h^\alpha +b_h \vartheta_h^{\alpha-1}+c_h\vartheta_h^\eta+\epsilon^{\BC}_h,
\end{equation}
which are listed alongside in the third column of Tables \ref{fhexplicit1} and \ref{fhexplicitC}. The functions $\epsilon^{\BC}_h$ are the terms involving $\lambda_{n+1}$ or $(1-\lambda_1)$ and converge to zero in the respective norms and help cancelling out specific error terms.

\begin{lemma}
\label{varthetaconvergence}
The approximate power functions $\vartheta_h^\alpha$, $\vartheta_h^{\alpha-1}$ and $\vartheta_h^{0}$ converge to $p_\alpha$, $p_{\alpha-1}$ and $p_0$, respectively in the $X$-norm and $\vartheta_h^{\alpha-2}$ converge to  $p_{\alpha-2}$ in $L_1[-1,1]$-norm. In particular, $f_h\to f$ for all $f\in \mathcal C(A,\BC)$ and corresponding $f_h$ in Tables \ref{fhexplicit1} and \ref{fhexplicitC}. Furthermore,
\begin{equation}\label{GthetaBeta}G_{\pm h}^{\BC}\vartheta^\beta_{\pm h}(x)=0 \end{equation} for $\iota(\pm x)<n$   and all relevant combinations of $\beta \in\{\alpha-1,0,\alpha-2\}$ and  boundary conditions given by Table \ref{explicitdomains} respectively. Finally,
\begin{equation}\label{GthetaAlpha}G_{- h}^{\mathrm{LN}}\vartheta^\alpha_{- h}(x)= 1, G_{+h}^{\mathrm{NN}}\vartheta^\alpha_{ h}(x)=1\end{equation} for $\iota(- x)<n$ and 
\begin{equation}\label{GthetaAlpha*}\lim_{h\to 0}\int_{-1}^{1-2h}\left|G_{+h}^{\mathrm{N*N}} \vartheta^\alpha_{+ h}(x)-1\right|\,dx=0.\end{equation}
\end{lemma}
\begin{proof}
On a fixed grid point $x = kh$, using \eqref{grunwaldproperties} observe that
$$\vartheta_h^{\beta}(kh-1) = h^{\beta} \mathcal{G}_{k-1}^{-\beta -1}=\frac{((k-1)h)^{\beta}}{\Gamma(\beta +1)} \left[1+ O ((k-1)^{-1})\right]= p_{\beta} (kh) + O (h).$$ 
As $\vartheta_h^\beta$ is interpolating between the grid points the total error converges to zero as well. 

The functions $\vartheta^\beta_{\pm h}$ were designed, using \eqref{grunwaldproperties}, such that \eqref{GthetaBeta}, \eqref{GthetaAlpha}, and \eqref{GthetaAlpha*} hold.  \end{proof}

\begin{table}
\centering
\vline
\begin{tabular}{l|l|l|}
  \hline
  \multicolumn{3}{c}{$X = L_1[-1,1]$ and $\mathcal{P}= I^\alpha P, \; P \in L_1[-1,1]$ polynomial }  
	\vline \\
	\hline
	 $(A^+, \BC)$ & $f \in \mathcal{C}(A^+, \BC)$ & $f_h \in L_1[-1,1] $\\
	\hline
	$(\partial^\alpha_+, \mathrm{DD})$ & $\mathcal{P} -  \frac{\mathcal{P}(1)}{p^+_{\alpha-1}(1)} p^+_{\alpha-1}$  &$\mathcal{P} - \frac{\mathcal{P}(1) }{p^+_{\alpha-1}(1)} \vartheta^{\alpha-1}_{+h}$ \\
  \hline
	$(\partial^\alpha_+, \mathrm{DN})$ &$ \mathcal{P} - \partial^{\alpha -1}_+ \mathcal{P}(1) p^+_{\alpha-1}$  & $ \left(\mathcal{P}  - \partial^{\alpha -1}_+\mathcal{P}(1)\vartheta_{+h}^{\alpha-1}\right)(1-\lambda_{n+1})$  \\
  \hline
	$(\partial^\alpha_+, \mathrm{ND})$  &  $\mathcal{P} -\mathcal{P}(1) p_0$ & $\mathcal{P} -\mathcal{P}(1) \vartheta^0_{+h}$ \\
  \hline
	$(\partial^\alpha_+, \mathrm{NN})$  &  $\mathcal{P}  - \frac{\partial^{\alpha -1}_+\mathcal{P}(1)}{p_1^+(1)} p^+_\alpha + cp_0$ & $\left(\mathcal{P}- \frac{\partial^{\alpha -1}_+\mathcal{P}(1)}{p_1^+(1)} \vartheta_{+h}^\alpha +c\vartheta^0_{+h}\right)(1-\lambda_{n+1}) $ \\
  \hline
	$(D^\alpha, \mathrm{ND})$   & $\mathcal{P} - \frac{ \mathcal{P} (1) }{p_{\alpha-2}^+(1)}p^+_{\alpha-2}$ & $\mathcal{P} - \frac{  \mathcal{P}(1)}{ p_{\alpha-2}^+ (1) }\vartheta_{+h}^{\alpha-2}$   \\
  \hline
	$(D^\alpha, \mathrm{NN})$   &$ \mathcal{P} - \frac{D^{\alpha -1} \mathcal{P}(1))}{p_1^+(1)} p_\alpha + cp_{\alpha -2}$ & $\left(\mathcal{P}-\frac{ D^{\alpha -1} \mathcal{P}(1)}{p^+_1(1)}\vartheta_h^\alpha+c\vartheta^{\alpha -2}_h\right)(1-\lambda_{n+1})$   \\
  \hline
  \end{tabular}
\caption{\label{fhexplicit1}
Functions $f_h \in L_1[-1,1]$ with $\lambda_{n+1}=\lambda$  if $\iota(x)= n+1$, and $0$ otherwise.}
\end{table}

\begin{table}
\centering
\vline
\begin{tabular}{l|l|l|}
  \hline
  \multicolumn{3}{c}{$X = C_0(\Omega)$ and $\mathcal{P}= I_-^\alpha (P-P(1)p_0), \; P \in C_0(\Omega)$ polynomial  }  
	\vline \\
	\hline
	$(A^-, \BC )$ & $f \in \mathcal{C}(A^-, \BC) $ & $f_h \in C_0(\Omega)$  \\
	\hline
	$(\partial^\alpha_-, \mathrm{DD})$ & $\mathcal{P} - \frac{\mathcal{P}(-1)}{p^-_{\alpha-1}(-1)} p^-_{\alpha-1}$ & $ \mathcal{P} - \frac{\mathcal{P}(-1)}{ \vartheta_{-h}^{\alpha-1}(-1)} \vartheta_{-h}^{\alpha-1}$  \\
  \hline
	$(\partial^\alpha_-, \mathrm{DN})$ &  $\begin{matrix}\mathcal{P} +P(1) p_\alpha^- - \hfill\hfill\\ (\mathcal{P}(-1)+P(1)p_{\alpha}^-(-1)) p_0\end{matrix}$ & $\begin{matrix}\mathcal{P}+ \frac{P(1)p_\alpha^-(-1)}{ \vartheta_{-h}^\alpha(-1)} \vartheta_{-h}^\alpha -\hfill\hfill\\ \quad (\mathcal{P}(-1)+P(1)p_\alpha^-(-1))\vartheta^0_h\end{matrix}$ \\
  \hline
	$(\partial^\alpha_-, \mathrm{ND})$ &$ \mathcal{P} - \partial^{\alpha -1}_- \mathcal{P}(-1) p^-_{\alpha-1}$  & $ \begin{matrix} \mathcal{P} -\frac{2 \partial^{\alpha -1}_-\mathcal{P}(-1)  -(\alpha - 1)h \partial^\alpha_- \mathcal{P}(-1)}{2} \vartheta_{-h}^{\alpha-1}  \\\quad  -h^\alpha \lambda'_1\partial_-^\alpha \mathcal P(-1)\hfill\hfill \end{matrix} $\\
  \hline
	$(\partial^\alpha_-, \mathrm{NN})$ 
	&
	  $ \mathcal{P}  - \frac{\partial^{\alpha -1}_-\mathcal{P}(-1)}{p^-_1(-1)} p^-_\alpha + cp_0$ 
	  & 
	  $ \begin{matrix}\mathcal{P}-\frac{ \partial^{\alpha -1}_-\mathcal{P}(-1) \left(1+\frac{h}{p_1^-(-1)}\right) - \frac{\alpha - 1}2h \partial^\alpha_- \mathcal{P}(-1)}{p^{-}_1(-1)} \vartheta_{-h}^\alpha \hfill\hfill\\\quad
	 +c\vartheta^0_h -h^\alpha\lambda'_1\left(\partial_-^\alpha\mathcal P(-1)-\frac{\partial_-^{\alpha-1}\mathcal P(-1)}{p_{1}^-(-1)}\right) \end{matrix} $ \\
  \hline
	$(\partial^\alpha_-, \mathrm{N^*D})$ & $\mathcal{P} + \frac{ \mathcal{P}' (-1)}{p_{\alpha-2}^-(-1)} p^-_{\alpha-1}$ & $\mathcal{P} +\frac{  \mathcal{P}'(-1)}{ p^-_{\alpha-2} (-1)  }\vartheta_{-h}^{\alpha-1}$\\
  \hline
	$(\partial^\alpha_-, \mathrm{N^*N})$ &$ \mathcal{P} + \frac{ \mathcal{P}'(-1)}{p_{\alpha-1}^-(-1)} p^-_\alpha + cp_0$ & $\mathcal{P}+ \frac{ \mathcal{P}'(-1) }{ p_{\alpha-1}^-(-1) } \vartheta_{-h}^\alpha+c \vartheta^0_h$\\
  \hline
\end{tabular}
\caption{\label{fhexplicitC}
Functions $f_h \in C_0(\Omega)$ with $\lambda'_1=1-\lambda$ if $\iota(x)= 1$, and $0$ otherwise.}
\end{table}

It remains to be shown that $G_{\pm h}^{\BC}\mathcal P$ converges on $\{x:\iota(\pm x)<n\}$ and that we have convergence of $G^{\BC}_{\pm h} f_h$ to $P+ap_0$ on $\{x:\iota(\pm x)\ge n\}$ as well. It follows from \cite[Theorem 5.1]{Baeumer2012a} (see, also, \cite[Proposition 4.9]{Baeumer2009}) that the shifted Gr\"unwald approximation formula
\begin{equation}
\label{grunwaldforboundeddomain1}
A^\alpha _{\pm h,q} p^\pm_\beta (x) = \frac {1}{h^\alpha} \sum_{k=0}^{N}  \mathcal{G}^\alpha _k p_{\beta}(x \mp(k-q))h)\to p^\pm_{\beta-\alpha}(x),
\end{equation}
converges on $C_0(\Omega)$ for $\beta>\alpha$, and in $L_1[-1,1]$ for $\beta>\alpha-1$. Here we can take $N =\iota(\pm (x+ph)) -1$ as $p^\pm_\beta(x)=0$ for $\pm x<-1$.

\begin{lemma}
In case of  $X=C_0(\Omega)$, $$\sup_{x:\iota(-x)<n}\left|G_{-h}^{\BC}\mathcal P(x)-P(x)+P(1)\right|\to 0$$
and in case of $X=L[-1,1]$,
$$\int_{-1}^{1-2h}\left|G_{+h}^{\BC}\mathcal P(x)-P(x)\right|\,dx\to 0.$$

\end{lemma}
\begin{proof}
This follows from the shifted Gr\"unwald formula \eqref{grunwaldforboundeddomain1} and the fact that $\mathcal P(x)=O(h^{\alpha+1})$ for $\iota(- x)\le 2$ and $X=C_0(\Omega)$, and $\mathcal P(x)=O(h^\alpha)$ for $\iota(x)\le 2$ and $X=L_1[-1,1]$.
\end{proof}

To be able to show convergence on the last two grids $\iota(\pm x)\ge n$, we use the boundary conditions to compensate for the deviation from the Gr\"unwald formula of the interpolation matrix.

\begin{lemma}In all cases of Tables \ref{fhexplicit1} and \ref{fhexplicitC}, for $f\in \mathcal C(A^+,\BC)$ or $f\in \mathcal C(A^-,\BC)$,  $$\int_{1-2h}^1\left|G_{+h}^{\BC} f_h(x)-A^+f(x)\right|\,dx\to 0\mbox{ or }\sup_{x\in[-1,-1+2h]}\left|G_{-h}^{\BC} f_h(x)-A^-f(x)\right|\to 0$$ respectively.\end{lemma}
\begin{proof} To obtain the approximate values of appropriate order of 
$\mathcal G^{\BC}_{\pm h}\mathcal P$ and of $\mathcal G^{\BC}_{\pm h}\vartheta_h^\beta$
on $\{x:\iota(\pm x)\ge n\}$ given in Tables \ref{ErrorTableL1} and \ref{ErrorTableC} we use the definition of $G_{\pm h}^{\BC}$, the Gr\"unwald formula, properties of the Gr\"unwald weights, and the Taylor expansion. 
For the Neumann boundary condition in the $C_0(\Omega)$ case  we also use the explicit error term of the Gr\"unwald formula for the $\alpha-1$ derivative. It is  given  in \cite{Baeumer2012a} in the first equation of the proof of Theorem 3.3 where the explicit coefficients of the Gr\"unwald multiplier $\omega_{q,\alpha-1}$ are given in equation (10) of \cite{Baeumer2012a}. In particular, $\beta=\alpha+n$, $n\ge 1$,
 \begin{equation}
\label{ErrortermsforgrunwaldformulaeA1}
A^{\alpha-1}_{h , q} p_\beta (x) =  p_{\beta-\alpha+1} (x) +h \left( q - \frac {\alpha-1} 2 \right)  p_{\beta-\alpha} (x) + O(h^2).
\end{equation}
Combining the approximate values according to Tables \ref{fhexplicit1} and \ref{fhexplicitC}  yields the result.
\end{proof}

Putting the three Lemmata of this section together finishes the proof of Proposition \ref{mainproposition}.  
\begin{table}
\centering
\vline
\begin{tabular}{c|c|c|}
	\hline
  \multicolumn{3}{c}{$X = L_1[-1,1]$ and $\mathcal{P}= I_+^\alpha P, \; P \in L_1[-1,1]$ polynomial  }  
	\vline \\
\hline
	 &$\iota(x)=n $&$\iota(x)=n+1$ \\
	\hline
	$G^{\mathrm{LD}}_{+h}\mathcal P$ &  $\begin{array}{l} P(x) -\dfrac { \lambda }{\alpha \lambda' +\lambda} \mathcal{P}(x+h)/h^\alpha\\ \quad\approx-\dfrac { \lambda }{\alpha \lambda' +\lambda} \mathcal P(1)/h^\alpha \end{array}$   & $\begin{array}{l} P(x)- \mathcal{P}(x+h)/h^\alpha\\ \quad\approx-\mathcal P(1)/h^\alpha\end{array} $ \\ 
  \hline
  $G^{\mathrm{LD}}_{+h}\vartheta^{\alpha-1}_{+h}$ & $-\dfrac{\lambda\vartheta_{+h}^{\alpha-1}(1-\lambda' h)}{\alpha\lambda'+\lambda}/h^\alpha\approx -\dfrac{\lambda p_{\alpha-1}^+(1)}{\alpha\lambda'+\lambda} /h^\alpha$   & $-\dfrac{\vartheta^{\alpha-1}_{+h}(1+\lambda h)}{h^\alpha} \approx-p_{\alpha-1}^+(1)/h^\alpha $ \\
  \hline
   $G^{\mathrm{LD}}_{+h}\vartheta^{0}_{+h}$ & $-\dfrac{\lambda}{\alpha\lambda'+\lambda}/h^\alpha$   & $-1/h^\alpha  $ \\ 
  \hline
    $G^{\mathrm{LD}}_{+h}\vartheta^{\alpha-2}_{+h}$ & $-\dfrac{\lambda\vartheta_{+h}^{\alpha-2}(1-\lambda' h)}{\alpha\lambda'+\lambda}/h^\alpha\approx -\dfrac{\lambda p_{\alpha-2}^+(1)}{\alpha\lambda'+\lambda} /h^\alpha$   & $-\dfrac{\vartheta^{\alpha-2}_{+h}(1+\lambda h)}{h^\alpha} \approx-p_{\alpha-2}^+(1)/h^\alpha $ \\
  \hline\hline
	$ G^{\mathrm{LN}}_{+h}\mathcal P$ &$\begin{array}{l}\lambda' P(x)-  \lambda\dfrac { A^{\alpha -1}_{h,0}\mathcal{P}(x)}{h} + \lambda \dfrac{\mathcal{P}(x+h)}{h^\alpha}
	  \\ \quad \approx -\lambda D^{\alpha-1}\mathcal{P}(1)/h  + \lambda \mathcal{P}(1)/h^\alpha \end{array}$  & $ \begin{array}{l}\dfrac { -\lambda'A^{\alpha -1}_{h,0}\mathcal{P}(x)}{h}-\lambda \mathcal P(x)/h^\alpha\\\quad \approx -\lambda' D^{\alpha-1}\mathcal P(1)/h-\lambda \mathcal P(1)/h^\alpha\end{array}
	$  \\
  \hline
    $G^{\mathrm{LN}}_{+h}\vartheta^{\alpha}_{+h}$ & $-\dfrac{\lambda p_1(1)}h+\lambda p_{\alpha}^+(1)/h^\alpha$   & $-\dfrac{\lambda'p_1(1)}{h} -\lambda p_{\alpha}^+(1)/h^\alpha $ \\ 
     \hline
    $G^{\mathrm{LN}}_{+h}\vartheta^{\alpha-1}_{+h}$ & $-\dfrac{\lambda}h+\lambda p_{\alpha-1}^+(1)/h^\alpha$   & $-\dfrac{\lambda'}{h} -\lambda p_{\alpha-1}^+(1)/h^\alpha $ \\ 
      \hline
    $G^{\mathrm{LN}}_{+h}\vartheta^{\eta}_{+h}$ & $\lambda p_{\eta}^+(1)/h^\alpha$   & $ -\lambda p_{\eta}^+(1)/h^\alpha $ \\ 
     \hline
    $G^{\mathrm{LN}}_{+h}\lambda_{n+1}$ & $\lambda/h^\alpha$   & $ -\lambda/h^\alpha $ \\ 
    \hline

\end{tabular}
\caption{\label{ErrorTableL1}
Approximate value of  $G_{+ h}^{\BC}\mathcal P+O(h^{1-\alpha})$  on the last two grids; i.e., $x\in[1-2h,1]$. Need to show that value of $G_{\pm h}^{\BC}f_h$ with $f_h$ of Table \ref{fhexplicit1} on these two grids is $o(h^{-1})$. We denote $\lambda'=1-\lambda$, $\eta\in\{0,\alpha-2\}$ and $\phi_h\approx \psi_h$ if $|\phi_h(x)-\psi_h(x)|h^{\alpha-1}\le M$.}
\end{table}

\begin{table}
\centering
\vline
\begin{tabular}{c|c|c|}

  \hline
  \multicolumn{3}{c}{$X = C_0(\Omega)$ and $\mathcal{P}= I_-^\alpha (P-P(1)p_0), \; P \in C_0(\Omega)$ polynomial  }  
	\vline \\
	\hline
	 &$\iota(x)=1 $&$\iota(x)=2$ \\
	\hline
	$\mathcal G^{ \mathrm{DR}}_{-h}\mathcal P$ & $\begin{array}{l} \dfrac{\alpha\lambda(P(x)-P(1))}{\lambda'+\alpha\lambda}-\dfrac{\alpha\lambda'\mathcal P(x)+\alpha\lambda \mathcal{P}(x-h)}{(  \lambda'+\alpha \lambda)h^\alpha}\\ \approx\quad
	-\dfrac{\alpha\lambda P(1)}{\lambda'+\alpha\lambda}-\dfrac{\alpha\mathcal P(-1)}{(  \lambda'+\alpha \lambda )h^\alpha}\end{array}$  & $P(x)-P(1)\approx -P(1)$ \\
  \hline

	$\mathcal G^{ \mathrm{DR}}_{-h}\vartheta^{\alpha}_{-h}$ & $\dfrac{\alpha\lambda}{\lambda'+\alpha\lambda}-\dfrac{\alpha}{  \lambda'+\alpha \lambda }\vartheta^{\alpha}_{-h}(-1)/h^\alpha$  & $1$ \\
  \hline
	$\mathcal G^{ \mathrm{DR}}_{-h}\vartheta^\beta_{-h}$ & $-\dfrac{\alpha}{  \lambda'+\alpha \lambda }\vartheta^\beta_{-h}(-1)/h^\alpha$  & $0$ \\
  
  \hline\hline
	$\mathcal G^{\mathrm{NR}}_{-h}\mathcal P$ &$\begin{array}{l}-A^{\alpha-1}_{-h,-\lambda}\mathcal P(-1)/h	\\\quad
	\approx -\dfrac{\partial_-^{\alpha-1}\mathcal P(-1)}{h}+
	(\lambda+\dfrac{\alpha-1}2)\partial_-^{\alpha}\mathcal P(-1)\end{array}
	$&$\begin{array}{l}\lambda \partial_-^{\alpha}\mathcal P(-1)-\dfrac{\lambda'A^{\alpha-1}_{-h,-\lambda-1}\mathcal P(-1)}{h}\\ \approx -\lambda'\dfrac{D_-^{\alpha-1}\mathcal P(-1)}{h}+\\\quad(1+\lambda'(\lambda+\dfrac{\alpha-1}{2}))\partial_-^{\alpha}\mathcal P(-1)\end{array}$ \\
  \hline
	$\mathcal G^{ \mathrm{NR}}_{-h}\vartheta^\alpha_{-h}$ & $-\dfrac{p^-_1(-1)}h+1+\lambda$  & $-\lambda'\dfrac{p^-_1(-1)}h+\lambda'(1+\lambda)+1$ \\
  \hline
	$\mathcal G^{ \mathrm{NR}}_{-h}\vartheta^{\alpha-1}_{-h}$ & $-1/h$  & $-\lambda'/h$ \\
  \hline
	$\mathcal G^{ \mathrm{NR}}_{-h}h^\alpha\lambda'_1$ & $-\lambda'$  & $\lambda\lambda'$ \\
  \hline
	$\mathcal G^{ \mathrm{NR}}_{-h}\vartheta^{0}_{-h}$ & $0$  & $0$ \\
	\hline\hline
	$\mathcal G^{\mathrm{N^*R}}_{-h}\mathcal P$   & $\begin{array}{l}\partial^\alpha_-\mathcal P(-1)-\dfrac{\mathcal P(x-h)-\mathcal P(x)}{h^\alpha}\\\quad\approx \partial^\alpha_-\mathcal P(-1)+\dfrac{\mathcal P'(-1)}{h^{\alpha-1}}\end{array}$& $ \partial^\alpha_-\mathcal P(-1)+\lambda'\dfrac{\mathcal P'(-1)}{h^{\alpha-1}}$\\
	 \hline
	$\mathcal G^{ \mathrm{N^*R}}_{-h}\vartheta^{\alpha}_{-h}$ & $1-p_{\alpha-1}^{-}(-1)/h^{\alpha-1}$  & $1-\lambda'p_{\alpha-1}^{-}(-1)/h^{\alpha-1}$ \\
	 \hline
	$\mathcal G^{ \mathrm{N^*R}}_{-h}\vartheta^{\alpha-1}_{-h}$ & $-1/h$  & $-\lambda'/h$ \\
  \hline
	$\mathcal G^{ \mathrm{N^*R}}_{-h}\vartheta^{0}_{-h}$ & $0$  & $0$ \\
  \hline
\end{tabular}
\caption{\label{ErrorTableC}
Approximate value of   $G_{- h}^{\BC}\mathcal P+O(h^{2-\alpha})$ on the first two grids; i.e. $x\in[-1,-1+2h)$. Need to show that value of $G_{- h}^{\BC}f_h$   with $f_h$ of Table \ref{fhexplicitC} on these two grids converges to $P(-1)-P(1)+a$, where $a$ is the coefficient of $p^\alpha_-$ in $f$.  We denote $\lambda'=1-\lambda$, $\beta\in\{\alpha-1,0\}$, and $\phi_h\approx \psi_h$ if $|\phi_h(x)-\psi_h(x)|/h^{2-\alpha}\le M$.}
\end{table}

\section*{Acknowledgements} We would like to thank Professor Mark Meerschaert for many fruitful discussions and support.

\section*{References}


\begin{thebibliography}{10}
\expandafter\ifx\csname url\endcsname\relax
  \def\url#1{\texttt{#1}}\fi
\expandafter\ifx\csname urlprefix\endcsname\relax\def\urlprefix{URL }\fi
\expandafter\ifx\csname href\endcsname\relax
  \def\href#1#2{#2} \def\path#1{#1}\fi

\bibitem{Engel2000}
K.-J. Engel, R.~Nagel, One-parameter semigroups for linear evolution
  equations., Graduate Texts in Mathematics. 194. Berlin: Springer. xxi, 586
  p., 2000.

\bibitem{Kallenberg1997}
O.~Kallenberg,
  \href{http://www.ebook.de/de/product/3249949/olav_kallenberg_foundations_of_modern_probability.html}{Foundations
  of Modern Probability}, Springer-Verlag GmbH, 2002.

\bibitem{Zhang2007b}
X.~Zhang, M.~Lv, J.~W. Crawford, I.~M. Young,
  \href{http://www.sciencedirect.com/science/article/pii/S0309170806002077}{The
  impact of boundary on the fractional advection--dispersion equation for
  solute transport in soil: Defining the fractional dispersive flux with the
  {C}aputo derivatives}, Advances in Water Resources 30~(5) (2007) 1205 --
  1217.
\newblock \href
  {http://dx.doi.org/http://dx.doi.org/10.1016/j.advwatres.2006.11.002}
  {\path{doi:http://dx.doi.org/10.1016/j.advwatres.2006.11.002}}.

\bibitem{Castillo-Negrete2006}
D.~del Castillo-Negrete, Fractional diffusion models of nonlocal transport,
  Physics of Plasmas 13~(8) (2006) 082308.
\newblock \href {http://dx.doi.org/10.1063/1.2336114}
  {\path{doi:10.1063/1.2336114}}.

\bibitem{Podlubny2009}
I.~Podlubny, A.~Chechkin, T.~Skovranek, Y.~Chen, B.~M.~V. Jara,
  \href{http://www.sciencedirect.com/science/article/pii/S0021999109000321}{Matrix
  approach to discrete fractional calculus ii: Partial fractional differential
  equations}, Journal of Computational Physics 228~(8) (2009) 3137 -- 3153.
\newblock \href {http://dx.doi.org/10.1016/j.jcp.2009.01.014}
  {\path{doi:10.1016/j.jcp.2009.01.014}}.

\bibitem{Du2012}
Q.~Du, M.~Gunzburger, R.~Lehoucq, K.~Zhou,
  \href{http://epubs.siam.org/doi/abs/10.1137/110833294}{Analysis and
  approximation of nonlocal diffusion problems with volume constraints}, SIAM
  Review 54~(4) (2012) 667--696.
\newblock 
  \href
  {http://dx.doi.org/10.1137/110833294} {\path{doi:10.1137/110833294}}.

\bibitem{Patie2017}
P.~Patie, Y.~Zhao, Spectral decomposition of fractional operators and a
  reflected stable semigroup, Journal of Differential Equations 262~(3) (2017)
  1690--1719.
\newblock \href {http://dx.doi.org/10.1016/j.jde.2016.10.026}
  {\path{doi:10.1016/j.jde.2016.10.026}}.

\bibitem{Patie2012}
P.~Patie, T.~Simon, Intertwining certain fractional derivatives, Potential
  Anal. 36 (2012) 569--587.

\bibitem{Baeumer2015}
B.~Baeumer, M.~Kov{\'{a}}cs, M.~M. Meerschaert, R.~L. Schilling, P.~Straka,
  Reflected spectrally negative stable processes and their governing equations,
  Transactions of the American Mathematical Society 368~(1) (2015) 227--248.
\newblock \href {http://dx.doi.org/10.1090/tran/6360}
  {\path{doi:10.1090/tran/6360}}.

\bibitem{Sankaranarayanan2014}
H.~Sankaranarayanan, \href{http://hdl.handle.net/10523/5216}{{G}r\"unwald-type
  approximations and boundary conditions for one-sided fractional derivative
  operators}, Ph.D. thesis, University of Otago, New Zealand (2014).

\bibitem{Bottcher2013}
B.~B\"ottcher, R.~Schilling, J.~Wang, L\'evy Matters III, 1st Edition, Vol.
  2099 of L\'evy Matters, Springer International Publishing, 2013.
\bibitem{Kolokoltsov2011}
V.~Kolokoltsov, {Markov
  Processes, Semigroups, and Generators}, De Gruyter studies in mathematics, De
  Gruyter, 2011.
  
\bibitem{Bertoin1996a}
J.~Bertoin, L\'evy processes, Vol. 121 of Cambridge Tracts in Mathematics,
  Cambridge University Press, Cambridge, 1996.

\bibitem{Bertoin1992}
J.~Bertoin, \href{http://dx.doi.org/10.1214/aop/1176989701}{An extension of
  {P}itman's {T}heorem for spectrally positive {L}\'evy processes}, Ann.
  Probab. 20~(3) (1992) 1464--1483.
\newblock \href {http://dx.doi.org/10.1214/aop/1176989701}
  {\path{doi:10.1214/aop/1176989701}}.

\bibitem{Folland1999}
G.~B. Folland,
  \href{http://www.ebook.de/de/product/3599843/gerald_b_folland_real_analysis.html}{Real
  Analysis}, John Wiley and Sons Ltd, 1999.

\bibitem{Haubold2011}
H.~J. Haubold, A.~M. Mathai, R.~K. Saxena, Mittag-{L}effler functions and their
  applications, Journal of Applied Mathematics 2011 (2011) Article ID 298628,
  51 pages.
\newblock \href {http://dx.doi.org/10.1155/2011/298628}
  {\path{doi:10.1155/2011/298628}}.

\bibitem{Arendt1986}
W.~Arendt, A.~Grabosch, G.~Greiner, U.~Groh, H.~P. Lotz,
  \href{http://www.ebook.de/de/product/7386061/wolfgang_arendt_annette_grabosch_guenther_greiner_ulrich_groh_heinrich_p_lotz_one_parameter_semigroups_of_positive_operators.html}{One-parameter
  Semigroups of Positive Operators}, Springer, 1986.

\bibitem{Oldham1974}
K.~B. Oldham, J.~Spanier,
  {Fractional
  Calculus (Mathematics in Science and Engineering, Vol. 111)}, Elsevier
  Science, 1974.

\bibitem{Tadjeran2006}
C.~Tadjeran, M.~M. Meerschaert, H.-P. Scheffler, A second-order accurate
  numerical approximation for the fractional diffusion equation, Journal of
  Computational Physics 213~(1) (2006) 205--213.
\newblock \href {http://dx.doi.org/10.1016/j.jcp.2005.08.008}
  {\path{doi:10.1016/j.jcp.2005.08.008}}.

\bibitem{Baeumer2012a}
B.~Baeumer, M.~Kov\'acs, H.~Sankaranarayanan, Higher order {G}runwald
  approximations of fractional derivatives and fractional powers of operators,
  Transactions of the American Mathematical Society 367 (2015) 813--834.
\newblock \href {http://dx.doi.org/10.1090/S0002-9947-2014-05887-X}
  {\path{doi:10.1090/S0002-9947-2014-05887-X}}.

\bibitem{Baeumer2009}B. Baeumer, M. Haase, M. Kov\'acs, Unbounded functional calculus for groups with applications, Journal of Evolution Equations 9 (2009) 171--195.
\newblock \href{http://dx.doi.org/10.1007/s00028-009-0012-z}{\path{doi:10.1007/s00028-009-0012-z}}




\bibitem{Arendt2001}
W.~Arendt, C.~J.~K. Batty, M.~Hieber, F.~Neubrander, Vector-valued {L}aplace
  transforms and {C}auchy problems, 2nd Edition, Vol.~96, Birkh\"auser Verlag,
  Basel, 2011.
\newblock \href {http://dx.doi.org/10.1007/978-3-0348-0087-7}
  {\path{doi:10.1007/978-3-0348-0087-7}}.


\end{thebibliography}

\end{document}